\documentclass[10pt,journal,compsoc]{IEEEtran}

\usepackage{amsmath,amssymb,amsfonts}
\usepackage{comment}
\usepackage{placeins}
\usepackage{graphicx}
\usepackage{cite}
\usepackage{url}
\usepackage{booktabs}
\usepackage{amsmath,amssymb}
\usepackage{tikz}
\usetikzlibrary{arrows.meta,positioning,calc}
\usepackage{multirow}
\usepackage{amsthm}
\newtheorem{theorem}{Theorem}
\newtheorem{proposition}{Proposition}

\usepackage{algorithm}
\usepackage{algpseudocode}
\theoremstyle{definition}
\newtheorem{definition}{Definition}
\newtheorem{remark}{Remark}

\newcommand{\TO}{\mbox{\textsc{TO}}}

\begin{document}

\title{
\textbf{SWAP}: A
\textbf{S}calable,
\textbf{W}arm-Startable,
\textbf{A}nytime \textbf{P}ermutation Solver
for Optimal Transport
}

\author{Xuekai~Jiang, 
        Zheng'ao~Liu,
        Lingyun~Qiu,
        Shenwen~Yu%
\IEEEcompsocitemizethanks{%
\IEEEcompsocthanksitem L. Qiu is with the Yau Mathematical Sciences Center,
Tsinghua University, and the Yanqi Lake Beijing Institute
of Mathematical Sciences and Applications, Beijing, China.
E-mail: lyqiu@tsinghua.edu.cn.
\IEEEcompsocthanksitem X. Jiang and Z. Liu are with Qiuzhen College,
Tsinghua University, Beijing, China.
E-mail: jxk24@mails.tsinghua.edu.cn and lza24@mails.tsinghua.edu.cn.
\IEEEcompsocthanksitem S. Yu is with the Department of Mathematical Sciences,
Tsinghua University, Beijing, China.
E-mail: ysw22@mails.tsinghua.edu.cn.}%
}

\pagestyle{plain}

\IEEEtitleabstractindextext{%
\begin{abstract}
Large-scale and high-dimensional discrete optimal transport often involves a tension between scalability and exact assignment feasibility: many scalable approaches optimize relaxed, factorized, or sparse transport plans rather than maintaining a one-to-one assignment throughout optimization. We introduce SWAP, an iterative solver that operates directly in permutation space for equally weighted point clouds. 
Starting from any feasible assignment, SWAP uses geometry-aware proposals to identify cost-decreasing local rearrangements. Consequently, every iterate is a valid permutation, so the one-to-one constraints are satisfied exactly throughout the optimization without rounding or feasibility correction. SWAP avoids constructing an  $N\times N$ cost or coupling matrix and requires only $O(Nd)$ memory.
We establish monotone descent, introduce a family of cycle-stability conditions, prove that SWAP reaches pairwise stability almost surely under generic nonparallel conditions, and derive sufficient conditions under which pairwise stability implies global optimality.
Experiments on an ImageNet assignment with 640,500 points in 2,048 dimensions, synthetic benchmarks, and MERFISH spatial alignment demonstrate that SWAP achieves lower transport costs on challenging large-scale and high-dimensional problems, with substantially reduced runtime and memory compared with existing OT solvers.
\end{abstract}

}

\maketitle
\thispagestyle{plain}
\raggedbottom

\IEEEdisplaynontitleabstractindextext
\IEEEpeerreviewmaketitle

\section{Introduction}
Optimal transport (OT) provides a mathematical framework for comparing probability distributions by seeking the most cost-efficient way to move mass from a source distribution to a target distribution. In the classical Monge formulation, transport is represented by a deterministic map that assigns each source location to a target location while transforming the source distribution into the target distribution at minimum cost. This formulation therefore naturally describes transport through point-to-point correspondences. The Kantorovich formulation relaxes this deterministic requirement by allowing mass to be distributed through a transport plan.

OT has found broad applications across statistics\cite{panaretosStatisticalAspects2019,ramdasWassersteinTwoSample2017,aguehBarycentersWasserstein2011,leonardSurveySchrodingerProblem2014}, machine learning\cite{genevayLearningGenerativeModels2018,courtyOptimalTransportDomain2017}, computer vision\cite{ferradansRegularizedDiscreteOptimal2014,bonneelDisplacementInterpolation2011}, generative modeling\cite{tongImprovingGeneralizingFlow2024}, scientific computing\cite{jordanVariationalFokkerPlanck1998,wibisonoSamplingOptimizationMeasures2018,bonneelDisplacementInterpolation2011,rabinWassersteinBarycenterTexture2011,taghvaeiOptimalTransportFPF2016}, and computational biology\cite{shenAccuratePointCloud2021,schiebingerOptimalTransportAnalysis2019}.
Among these applications, two complementary computational demands are particularly relevant to modern discrete OT: scalability to large and high-dimensional empirical distributions, and the ability to recover explicit one-to-one assignments when required.

First, many modern applications involve large-scale or high-dimensional empirical distributions, and often both. Representative examples include single-cell perturbation analysis, cell-state and multimodal single-cell alignment, large-scale developmental trajectory reconstruction, cross-modal retrieval, and domain adaptation
\cite{bunneLearningSinglecellPerturbation2023,
dongCausalIdentificationSinglecell2023,
kleinMappingCellsTime2025,
samaranScConfluenceSinglecellDiagonal2024,
hanLearningRematchMismatched2024,
liuCOTUnsupervisedDomain2023}.
In such settings, the number of samples, the ambient feature dimension, or both can be large, making the computational and memory costs of OT a central practical concern.

Second, in many matching and alignment applications, the desired output is not merely a transport cost or a soft coupling, but an explicit one-to-one correspondence. Such discrete matching structures arise in point-cloud scene flow, shape correspondence, bilingual lexicon induction, and entity alignment
\cite{puyFLOT2020,
liSelfPointFlowSelfSupervisedScene2021,
marchisioBilingualLexiconInduction2022a,
dingConflictAwarePseudoLabeling2022,
leIntegratingEfficientOptimal2024}.
Permutation-valued formulations also arise in other applications, such as the inference of latent event orderings in disease progression
\cite{wijeratneUnscramblingDiseaseProgression2024a}.
Although these applications may ultimately require discrete assignments or permutations, many existing OT-based approaches operate through relaxed or regularized transport plans, from which hard correspondences are extracted when needed.

Existing computational approaches address these problems through exact optimization, regularization, and structured approximation. Classical discrete OT is a transportation linear program that can be solved by simplex or network-simplex methods
\cite{dantzigApplicationSimplex1951,orlinPolynomialTimePrimal1997}.
For equally weighted measures with the same number of atoms, it reduces to a linear assignment problem, for which the Hungarian and Jonker--Volgenant algorithms return globally optimal permutations
\cite{kuhnHungarianMethodAssignment1955,jonkerShortestAugmentingPath1987}.
Other high-accuracy approaches include accelerated gradient, primal--dual, and splitting methods
\cite{dvurechenskyComputationalOptimalTransport2018,
jambulapatiDirectParallelOT2019,
chambolleAcceleratedBregmanPrimalDual2022,
maiFastAccurateSplitting2021},
while entropic regularization leads to the computationally attractive Sinkhorn algorithm
\cite{cuturiSinkhornDistancesLightspeed2013}.

A broad range of scalable methods further reduce the computational burden by exploiting approximation or additional problem structure. These include multiscale and hierarchical schemes
\cite{merigotMultiscaleApproachOptimal2011,schmitzerHierarchicalApproachOptimal2013},
factored and low-rank representations
\cite{forrowStatisticalOptimalTransport2019,linMakingTransportMore2021,scetbonLowRankOptimalTransport2022},
minibatch approaches
\cite{fatrasLearningMinibatchWasserstein2020,
fatrasUnbalancedMinibatchOptimal2021,
nguyenImprovingMinibatchOptimal2022},
projection-based methods
\cite{nguyenHierarchicalSlicedWasserstein2023,
mengLargescaleOptimalTransport2019},
and neural parameterizations of transport maps or potentials
\cite{makkuvaOptimalTransportMapping2020,
korotinNeuralOptimalTransport2021,
buzunExpectileRegularizationFast2024}.
Specialized geometric or algebraic structure can also yield substantial reductions in complexity
\cite{luoConeCompatibleMonge2026,zhangHOTEfficientHalpern2025}.
These approaches have greatly expanded the range of tractable OT problems, but scalability is often achieved through regularization, compression, surrogate subproblems, or additional structural assumptions. Moreover, methods formulated through transport plans generally do not maintain an explicit one-to-one assignment throughout the optimization. This leaves a gap for large-scale and high-dimensional discrete OT problems in which a feasible permutation is itself the desired output.

In this work, we consider the large-scale discrete Monge optimal transport
problem between two equally weighted point clouds
\[
X=\{x_i\}_{i=1}^N,\qquad
Y=\{y_j\}_{j=1}^N,
\]
where \(x_i,y_j\in\mathbb{R}^d\). The corresponding empirical measures are
defined as
\[
\mu=\frac{1}{N}\sum_{i=1}^N\delta_{x_i},
\qquad
\nu=\frac{1}{N}\sum_{j=1}^N\delta_{y_j}.
\]
For a given transport cost \(c(x,y)\), the equal-weight Monge
formulation seeks a permutation
\[
\sigma\in S_N
\]
that minimizes the total matching cost
\[
\min_{\sigma\in S_N}
C_c(\sigma)
:=
\sum_{i=1}^N c(x_i,y_{\sigma(i)}).
\]
The associated discrete optimal transport cost is then given by
\[
\mathcal{W}_c(\mu,\nu)
=
\frac{1}{N}
\min_{\sigma\in S_N}
\sum_{i=1}^N c(x_i,y_{\sigma(i)}).
\]
The quadratic cost
\[
c(x,y)=\|x-y\|^2
\]
corresponds to the classical empirical squared Wasserstein-2 distance
as a special case.

Our goal is to solve this permutation-valued transport problem directly
at scales for which conventional dense OT solvers become prohibitively
expensive. To this end, we propose \emph{SWAP}, an iterative framework
that operates entirely in the space of permutations. Starting from any
feasible assignment, SWAP repeatedly identifies local rearrangements
that decrease the original transport cost. To make this search efficient
at large scale, random one-dimensional projections are used to generate
structured candidate pairs, while all accepted updates are verified using
the original high-dimensional cost. Consequently, the algorithm avoids
forming a dense cost or coupling matrix, preserves a valid one-to-one
assignment throughout the optimization, and can be stopped or resumed at
any iteration.
Our main contributions are summarized as follows:
\begin{itemize}
    \item We propose \emph{SWAP}, which is designed for large-scale and high-dimensional transport problems. By generating structured candidate pairs through one-dimensional projections and evaluating only the proposed exchanges in the original space, the algorithm requires $O(L(Nd+N\log N))$ operations with \(L\) sliced directions and $O(Nd)$ memory, enabling efficient computation on datasets with very large $N$ and high ambient dimension $d$.
    
    \item SWAP is an iterative permutation-based method that is both warm-startable and anytime. Every iterate produced by SWAP is a valid permutation, so the algorithm can be terminated at any stage and still return a feasible one-to-one assignment without an additional rounding step. Moreover, any feasible permutation produced by another method can be used as an initialization and further refined under the original transport objective.

   \item We establish almost-sure convergence of SWAP to pairwise stability under generic nonparallel conditions and derive sufficient conditions under which pairwise stability implies global optimality.
\end{itemize}

The remainder of this paper is organized as follows.
Section~2 reviews related computational methods for discrete optimal
transport. Section 3 introduces SWAP and analyzes its computational complexity. Section 4 develops the cycle-stability, convergence, and optimality theory. Section 5 presents the numerical experiments. Section 6 concludes the paper.
\section{Related Work}

Existing methods for discrete optimal transport exhibit different
tradeoffs between exactness, scalability, memory consumption, and the
structure of the returned solution. We focus on three families of methods that are most closely related to the computational setting considered in this work.

\noindent\textbf{Exact Assignment and LP-Based Solvers.}
Classical discrete OT can be formulated as a transportation linear
program and solved using simplex, network-simplex, or minimum-cost-flow
methods
\cite{dantzigApplicationSimplex1951,
tarjanDynamicTreesSearch1997,
orlinPolynomialTimePrimal1997,
huangfuParallelizingDualRevised2018}.
For equally weighted point clouds of equal cardinality, the problem
reduces to a linear assignment problem. The Hungarian and
Jonker--Volgenant algorithms directly compute globally optimal
permutations
\cite{kuhnHungarianMethodAssignment1955,
jonkerShortestAugmentingPath1987}.
The Hungarian algorithm has worst-case time complexity $O(N^3)$, while
forming and storing the dense cost matrix requires $O(N^2)$ memory.
Accelerated gradient, primal--dual, and splitting methods provide other
high-accuracy approaches to the unregularized transport problem
\cite{dvurechenskyComputationalOptimalTransport2018,
jambulapatiDirectParallelOT2019,
chambolleAcceleratedBregmanPrimalDual2022,
maiFastAccurateSplitting2021}.
Although these methods solve the desired discrete problem directly, their
computational and memory costs make them difficult to apply to very large
point clouds.

\noindent\textbf{Entropically Regularized OT Method.}
Entropic regularization replaces the original discrete OT problem by a
smooth regularized problem that can be solved efficiently using
Sinkhorn matrix-scaling iterations
\cite{cuturiSinkhornDistancesLightspeed2013,
ferradansRegularizedDiscreteOptimal2014}.
Sinkhorn is often substantially more practical than exact assignment at
moderate problem sizes. Existing complexity analyses improve the
dependence required for an $\varepsilon$-accurate approximation of
unregularized OT from
$\widetilde{O}(N^2/\varepsilon^3)$ to
$\widetilde{O}(N^2/\varepsilon^2)$
\cite{dvurechenskyComputationalOptimalTransport2018,
luoImprovedComplexityAnalysis2023},
with related results for Greenkhorn and accelerated, stabilized, and
progressive variants
\cite{linEfficientOptimalTransport2019,
linEfficiencyEntropicRegularized2022,
guminovCombinationAlternatingMinimization2021,
schmitzerStabilizedSparseScaling2019,
kassraieProgressiveEntropicOptimal2024}.
However, classical dense Sinkhorn still requires an $N\times N$ cost or
kernel matrix, leading to $O(N^2)$ memory consumption and quadratic
work in the dominant matrix operations. Moreover, it solves a
regularized problem and generally returns a dense soft coupling rather
than an explicit permutation. When a hard one-to-one assignment is
required, an additional discretization or assignment step is therefore
needed.

\noindent\textbf{Sparse and Hierarchical OT Methods.}
Another line of scalable OT methods reduces the effective size of the transport problem by exploiting sparsity, multiscale organization, or problem-specific structure. Multiscale approaches solve transport problems over progressively refined representations, while sparse schemes restrict computation to selected source--target pairs rather than explicitly maintaining the full dense coupling
\cite{merigotMultiscaleApproachOptimal2011,
gerberMultiscaleStrategies2017,
schmitzerSparseMultiscaleAlgorithm2016}.

Hierarchical Refinement (HiRef)
\cite{halmosHierarchicalRefinementOptimal2025}
combines low-rank OT factorizations
\cite{scetbonLowRankSinkhornFactorization2021,
scetbonLowRankOptimalTransport2022,
halmosLowRankFactorRelaxation2024}
with recursive multiscale refinement. It uses low-rank transport subproblems to construct a hierarchy and progressively refines the resulting source and target clusters; for equally weighted point clouds of equal cardinality, this process ultimately produces a bijective source--target correspondence. For fixed refinement rank $r$ and algebraic cost rank $d_C$, HiRef has time complexity
$\Theta(Nd_Cr\log_r N)$
and linear memory complexity. For squared Euclidean costs in
$\mathbb{R}^d$, $d_C\leq d+2$, so its computational cost retains an explicit dependence on the ambient dimension. Moreover, its exact-recovery guarantee relies on globally solving the underlying low-rank factorized subproblems, which are generally nonconvex.

HALO \cite{xiaMemoryEfficientHierarchical2026} follows a different hierarchical strategy, combining multiscale discretization with sparse active-support refinement. Rather than representing all possible source--target interactions, it progressively maintains a restricted set of candidate pairs using multiscale information, geometric shielding, and dual-feasibility corrections
\cite{merigotMultiscaleApproachOptimal2011,
gerberMultiscaleStrategies2017,
schmitzerSparseMultiscaleAlgorithm2016}.
The resulting sparse transport problems are solved using factorization-free first-order techniques designed for large-scale parallel computation
\cite{applegatePracticalLargeScaleLP2021}.
For squared Euclidean costs, HALO can maintain an $O(N)$ active support under suitable geometric regularity and stability assumptions, thereby avoiding explicit storage of a dense transport plan. Its efficiency, however, relies on geometric structure that becomes more difficult to exploit as the ambient dimension increases.

A different structure-exploiting approach is HOT
\cite{zhangHOTEfficientHalpern2025}, which is designed for large-scale two-dimensional histogram OT with squared Euclidean ground cost. By exploiting the separability of the cost, HOT reformulates the original problem as an equivalent reduced linear program and solves it using a Halpern-accelerated first-order method together with a structure-exploiting linear-system solver. For $N=mn$ supports with $m=n$, it achieves
$O(N^{3/2}/\varepsilon)$ computational complexity and
$O(N^{3/2})$ memory complexity for an $\varepsilon$-approximate solution, from which an optimal transport plan for the original problem can be reconstructed. This reduction, however, relies on the special two-dimensional grid structure and squared Euclidean cost.

These methods achieve scalability by reducing or structuring the transport-plan representation through low-rank factorization, sparse active supports, hierarchical refinement, or specialized geometry. In contrast, SWAP does not explicitly construct or progressively refine a coupling-matrix representation. It operates directly in the permutation space, maintains an explicit feasible one-to-one assignment throughout optimization, and evaluates only structured candidate rearrangements under the original transport objective.

To summarize the discussion above, Table~\ref{tab:method-comparison}
compares the representative methods from four perspectives that are
particularly relevant to large-scale discrete Monge OT: whether the
method returns a valid permutation, the practical dimension regime, computational complexity, and
memory complexity. The reported memory complexities include storage of
the input point clouds, while the computational complexities include
the cost of evaluating the required pairwise distances when applicable.
For HOT and Sinkhorn, $\varepsilon$ denotes the target solution accuracy.
\begin{table*}[t]
\centering
\caption{Comparison of representative discrete optimal transport methods.
Here, $N$ denotes the number of points in each point cloud, $d$ the ambient
dimension, $r$ the refinement rank, $d_C$ the rank of the cost-matrix factorization used in HiRef and $L$ the total sliced directions used in SWAP. For squared Euclidean costs,
$d_C\le d+2$.}
\label{tab:method-comparison}
\begin{tabular}{lcccc}
\toprule
Method
& Valid permutation
& Practical dimensional regime
& Computational complexity
& Memory complexity \\
\midrule

Hungarian
& $\checkmark$
& High
& $O(N^2d+N^3)$
& $O(Nd+N^2)$ \\

Sinkhorn
& $\times$
& Medium
& $\tilde O(N^2/\varepsilon^2)$
& $O(Nd+N^2)$ \\

HOT
& $\times$
& $d=2$
& $O(N^{3/2}/\varepsilon)$
& $O(Nd+N^{3/2})$ \\

HiRef
& $\checkmark$
& High
& $O(Nd_Cr\log_r N)$
& $O(Nd+Nr)$ \\ 

HALO
& $\times$
& Medium
& Not specified
& $O(Nd)$ \\

SWAP
& $\checkmark$
& High
& $O\!\left(L(Nd+N\log N)\right)$
& $O(Nd)$ \\

\bottomrule
\end{tabular}
\end{table*}
\section{SWAP: Projection-Guided Permutation Refinement}

Our method is motivated by the cyclic monotonicity characterization of
optimal transport. For a permutation-valued transport plan, a cycle of source
indices represents a local rearrangement of the current matching: the targets
assigned to the selected source points are cyclically exchanged while all
other assignments remain unchanged. If such a rearrangement strictly decreases
the transport cost, we call the corresponding cycle a \emph{decreasing cycle}.

The absence of decreasing cycles of all possible lengths characterizes global
optimality. We recall the cyclic monotonicity characterization for discrete OT, which is proved in
\cite{villani2003topics}:
\begin{theorem}[Cyclic monotonicity for discrete optimal transport]
Let
\[
\mu=\frac1N\sum_{i=1}^N\delta_{x_i},
\qquad
\nu=\frac1N\sum_{i=1}^N\delta_{y_i},
\]
and let \(\sigma\in S_N\) be a matching permutation. Let \(c(x,y)\) be a
finite transport cost. Then \(\sigma\) is globally optimal if and only if its
matched set
\[
\{(x_i,y_{\sigma(i)})\}_{i=1}^N
\]
is \(c\)-cyclically monotone, namely, for every \(m\ge 2\) and every
collection of distinct indices \(i_1,\ldots,i_m\), with
\(i_{m+1}=i_1\), one has
\begin{equation}\label{cycle condition}
    \sum_{r=1}^m c(x_{i_r},y_{\sigma(i_r)})
    \le
    \sum_{r=1}^m c(x_{i_{r+1}},y_{\sigma(i_r)}).
\end{equation}
\end{theorem}

The case \(m=2\) in \eqref{cycle condition} gives
\[
c(x_{i_1},y_{\sigma(i_1)})
+c(x_{i_2},y_{\sigma(i_2)})
\leq
c(x_{i_1},y_{\sigma(i_2)})
+c(x_{i_2},y_{\sigma(i_1)}),
\]
which is the standard pairwise monotonicity condition. Therefore, a
decreasing cycle of length two corresponds exactly to a cost-decreasing pairwise exchange, which is referred to as a decreasing pair in this paper.
This special case will play a central role in the design and analysis of our
algorithm.

\subsection{Basic idea of SWAP}
We now introduce the main algorithm, SWAP. The basic idea of SWAP is to iteratively test whether exchanging
the target assignments of two source points can decrease the transport cost.
Let \(\sigma\in S_N\) be the current matching permutation, where \(x_i\) is
matched to \(y_{\sigma(i)}\). For two distinct indices
\(i,j\in\{1,\dots,N\}\), the \(i,j\)-swap of \(\sigma\) is the permutation
\(\sigma^{(i,j)}\in S_N\) defined by
\[
\sigma^{(i,j)}(i)=\sigma(j),
\qquad
\sigma^{(i,j)}(j)=\sigma(i),
\]
and
\[
\sigma^{(i,j)}(k)=\sigma(k),
\qquad
k\neq i,j.
\]
Thus, the swap exchanges the two target assignments of \(x_i\) and \(x_j\)
while leaving all other assignments unchanged. The \(i,j\)-swap decreases the
cost if and only if
\[
c(x_i,y_{\sigma(j)})+c(x_j,y_{\sigma(i)})
<
c(x_i,y_{\sigma(i)})+c(x_j,y_{\sigma(j)}).
\]
We refer to this inequality as the \emph{decreasing condition}, and we call a pair
\((i,j)\) satisfying it a \emph{decreasing pair},  which corresponds to a decreasing cycle of length two. For the quadratic cost, the preceding inequality is
equivalent to
\[
\langle x_i-x_j,\;y_{\sigma(i)}-y_{\sigma(j)}\rangle<0.
\]
This inner-product form is therefore a special simplification of the general
decreasing condition.
\subsection{Projection-induced owner proposals}
 Since exhaustively testing all $O(N^2)$ pairs is computationally prohibitive for large-scale problems, the central challenge is to identify decreasing pairs efficiently without enumerating all possible pairs. In this subsection, we introduce a sliced search strategy for identifying decreasing pairs. The main idea is motivated by one-dimensional optimal transport, where the optimal matching is characterized by sorting.

\begin{proposition}[Monotone rearrangement in one dimension]
Let
\[
\mu_N=\frac{1}{N}\sum_{i=1}^N\delta_{x_i},
\qquad
\nu_N=\frac{1}{N}\sum_{i=1}^N\delta_{y_i},
\]
where $x_i,y_i\in\mathbb R$. Assume that
\[
c(x,y)=h(|x-y|)
\]
satisfies the Monge condition
\[
c(x_1,y_1)+c(x_2,y_2)
\leq
c(x_1,y_2)+c(x_2,y_1)
\]
whenever $x_1\leq x_2$ and $y_1\leq y_2$. In particular, this holds when $h$ is nondecreasing and convex. Let $\sigma_1,\sigma_2\in S_N$ satisfy
\[
x_{\sigma_1(1)}
\leq\cdots\leq
x_{\sigma_1(N)}
\]
and
\[
y_{\sigma_2(1)}
\leq\cdots\leq
y_{\sigma_2(N)}.
\]
Then an optimal matching is given by
\[
x_{\sigma_1(r)}
\mapsto
y_{\sigma_2(r)},
\qquad
r=1,\ldots,N.
\]
\end{proposition}

Thus, in one dimension, the equally weighted transport problem can be solved by sorting the two point clouds and matching points with the same rank. Although no such global ordering is available in higher dimensions, random one-dimensional projections provide useful ordering information for candidate pairs.

We first sample a random direction
\[
\theta\in\mathbb S^{d-1}
\]
and project the source and target point clouds onto this direction:
\[
s_i=\langle x_i,\theta\rangle,
\qquad
t_j=\langle y_j,\theta\rangle.
\]
Let $\sigma_1,\sigma_2\in S_N$ denote the corresponding sorting permutations:
\[
s_{\sigma_1(1)}
\leq
s_{\sigma_1(2)}
\leq\cdots\leq
s_{\sigma_1(N)}
\]
and
\[
t_{\sigma_2(1)}
\leq
t_{\sigma_2(2)}
\leq\cdots\leq
t_{\sigma_2(N)}.
\]
The one-dimensional monotone rearrangement then induces the temporary assignment
\[
x_{\sigma_1(r)}
\mapsto
y_{\sigma_2(r)},
\qquad
r=1,\ldots,N.
\]
Equivalently, we define the sorting-induced permutation $\tau_\theta\in S_N$ by
\[
\tau_\theta(\sigma_1(r))
=
\sigma_2(r),
\qquad
r=1,\ldots,N.
\]

The current matching is represented by $\sigma$, namely,
\[
x_i\mapsto y_{\sigma(i)}.
\]
If
\[
\tau_\theta(i)=\sigma(i),
\]
then the projected ordering agrees with the current assignment of $x_i$.
Otherwise, let
\[
j=\sigma^{-1}\bigl(\tau_\theta(i)\bigr).
\]
Since
\[
\sigma(j)=\tau_\theta(i),
\]
the source point $x_j$ currently owns the target suggested for $x_i$ by
the sliced matching. We therefore call $j$ the \emph{current owner} of
$\tau_\theta(i)$. The pair $(i,j)$ is then generated as a candidate pair; see
Fig.~\ref{fig:swap-pair}.

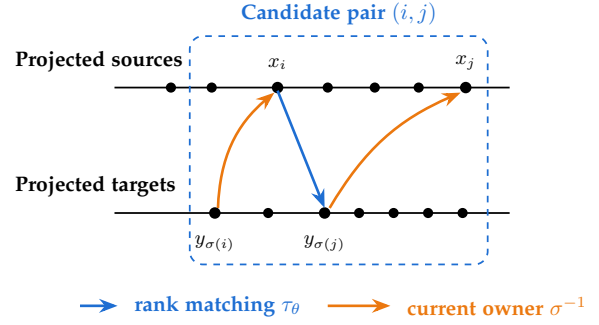
\begin{figure}[htbp]
\centering
\resizebox{0.88\columnwidth}{!}{%
\begin{tikzpicture}[
    >=Stealth,
    every node/.style={font=\small}
]

\definecolor{rankblue}{RGB}{30,110,210}
\definecolor{ownerorange}{RGB}{235,120,15}

\node[
    anchor=west,
    font=\bfseries
]
at (0.15,2.42)
{Projected sources};

\node[
    anchor=west,
    font=\bfseries
]
at (0.15,0.42)
{Projected targets};

\begin{scope}[xshift=1.60cm]

\draw[thick] (0.25,2) -- (6.55,2);
\draw[thick] (0.25,0) -- (6.55,0);

\fill (1.15,2) circle (2.3pt);
\fill (1.80,2) circle (2.3pt);

\coordinate (xi) at (2.85,2);
\coordinate (xj) at (5.85,2);

\fill (xi) circle (2.6pt);
\fill (xj) circle (2.6pt);

\node[above=5pt] at (xi) {$x_i$};
\node[above=5pt] at (xj) {$x_j$};

\fill (3.65,2) circle (2.3pt);
\fill (4.40,2) circle (2.3pt);
\fill (5.10,2) circle (2.3pt);

\coordinate (ysi) at (1.85,0);

\coordinate (ysj) at (3.60,0);

\fill (ysi) circle (2.6pt);
\fill (ysj) circle (2.6pt);

\node[below=7pt] at (ysi) {$y_{\sigma(i)}$};
\node[below=7pt] at (ysj) {$y_{\sigma(j)}$};

\fill (2.70,0) circle (2.3pt);

\fill (4.15,0) circle (2.3pt);
\fill (4.70,0) circle (2.3pt);
\fill (5.25,0) circle (2.3pt);
\fill (5.80,0) circle (2.3pt);

\draw[
    dashed,
    thick,
    rankblue,
    rounded corners=6pt
]
(1.45,-0.82) rectangle (6.20,2.82);

\node[
    rankblue,
    font=\bfseries
]
at (3.83,3.18)
{Candidate pair $(i,j)$};

\draw[
    ->,
    very thick,
    rankblue
]
($(xi)+(0,-0.05)$)
--
($(ysj)+(0,0.08)$);

\draw[
    ->,
    very thick,
    ownerorange
]
($(ysi)+(0.05,0.08)$)
to[bend left=24]
($(xi)+(-0.05,-0.08)$);

\draw[
    ->,
    very thick,
    ownerorange
]
($(ysj)+(0.08,0.08)$)
to[bend left=18]
($(xj)+(-0.08,-0.08)$);

\draw[
    ->,
    very thick,
    rankblue
]
(-0.30,-1.48) -- (0.30,-1.48);

\node[
    anchor=west,
    rankblue,
    font=\bfseries
]
at (0.45,-1.48)
{rank matching $\tau_\theta$};

\draw[
    ->,
    very thick,
    ownerorange
]
(3.65,-1.48) -- (4.65,-1.48);

\node[
    anchor=west,
    ownerorange,
    font=\bfseries
]
at (4.80,-1.48)
{current owner $\sigma^{-1}$};

\end{scope}

\end{tikzpicture}
}
\caption{
Illustration of the sliced generation of a candidate pair $(i,j)$.
}
\label{fig:swap-pair}
\end{figure}
Each generated candidate pair is subsequently tested using the original
high-dimensional decreasing condition, and only cost-decreasing pairs are accepted.
\begin{algorithm}[htbp]
\caption{SWAP with sliced projection search}
\label{alg:swap}
\begin{algorithmic}

\Require Source distribution $\mu$, target distribution $\nu$,
number of points $N$, initialization strategy $\mathcal{I}$

\Statex

\State \textbf{Step 1: Initialization.}
Obtain two equally weighted point clouds
\[
X=\{x_i\}_{i=1}^N,
\qquad
Y=\{y_j\}_{j=1}^N
\]
from $\mu$ and $\nu$ using a chosen sampling or discretization procedure,
and generate an initial permutation
\[
\sigma^{(0)}
\gets
\mathcal{I}(X,Y),
\qquad
\sigma^{(0)}\in S_N.
\]
Set $k\gets 0$.

\Statex

\While{the stopping criterion is not satisfied}

    \State \textbf{Step 2: Sliced search.}

    Define the current target assignment
    \[
    z_i=y_{\sigma^{(k)}(i)},
    \qquad
    i=1,\ldots,N.
    \]

    Sample a random projection direction
    \[
    \theta_k
    \sim
    \operatorname{Unif}(\mathbb S^{d-1}).
    \]

    Sort the source and current target point clouds according to their
    projections:
    \[
    \langle\theta_k,x_{i_1}\rangle
    \leq\cdots\leq
    \langle\theta_k,x_{i_N}\rangle
    \]
    and
    \[
    \langle\theta_k,z_{j_1}\rangle
    \leq\cdots\leq
    \langle\theta_k,z_{j_N}\rangle.
    \]

    Generate the candidate pairs
    \[
    \mathcal{C}_k
    =
    \left\{
    (i_r,j_r):
    i_r\neq j_r
    \right\}.
    \]

    \Statex

    \State \textbf{Step 3: Test the decreasing condition.}

    Set
    \[
    \sigma^{(k+1)}
    \gets
    \sigma^{(k)}.
    \]

    \For{each $(i,j)\in\mathcal{C}_k$}

        Compute
        \[
        \begin{aligned}
        \Delta_{ij}
        ={}&
        c\bigl(x_i,y_{\sigma^{(k+1)}(j)}\bigr)
        +
        c\bigl(x_j,y_{\sigma^{(k+1)}(i)}\bigr)\\
        &-
        c\bigl(x_i,y_{\sigma^{(k+1)}(i)}\bigr)
        -
        c\bigl(x_j,y_{\sigma^{(k+1)}(j)}\bigr).
        \end{aligned}
        \]

        \If{$\Delta_{ij}<0$}
            Swap the two assigned targets:
            \[
            \sigma^{(k+1)}(i)
            \leftrightarrow
            \sigma^{(k+1)}(j).
            \]
        \EndIf

    \EndFor

    \State $k\gets k+1$.

\EndWhile

\Statex
\State \Return $\sigma^{(k)}$

\end{algorithmic}
\end{algorithm}
The sliced search can therefore be summarized as
\[
\begin{aligned}
&\text{Sliced projection}
\Longrightarrow
\text{Generate candidate pairs}\\
&\Longrightarrow
\text{Full-dimensional decreasing condition test}
\end{aligned}
\]

The complete algorithm is summarized in Algorithm~\ref{alg:swap}.

The sliced projection strategy naturally extends the decreasing pair search
to the identification of decreasing cycles of larger length. Using the
current-owner relation defined above, each source index $i$ induces the
transition
\[
i
\longmapsto
\sigma^{-1}\bigl(\tau_\theta(i)\bigr).
\]
Starting from an index $i_1$ and repeatedly applying this transition generates
an owner chain.

More precisely, define
\[
j_\ell=\tau_\theta(i_\ell),
\qquad
i_{\ell+1}=\sigma^{-1}(j_\ell),
\qquad
\ell\geq1.
\]
This produces the alternating sequence
\[
x_{i_1}
\longrightarrow
y_{j_1}
\longrightarrow
x_{i_2}
\longrightarrow
y_{j_2}
\longrightarrow
\cdots,
\]
where $y_{j_\ell}$ is the target suggested for $x_{i_\ell}$ by the sliced
matching, and $x_{i_{\ell+1}}$ is its current owner. The construction terminates when either \(\ell=m\) or \(i_{\ell+1}=i_1\). 

When this construction closes after two source indices, it reduces to the
swap-pair mechanism used in the sliced search. More generally, an owner chain involving at most $m$ source points defines a localized block in which an improving
cycle of length up to $m$ can be identified efficiently; see
Fig.~\ref{fig:owner-chain-block}.

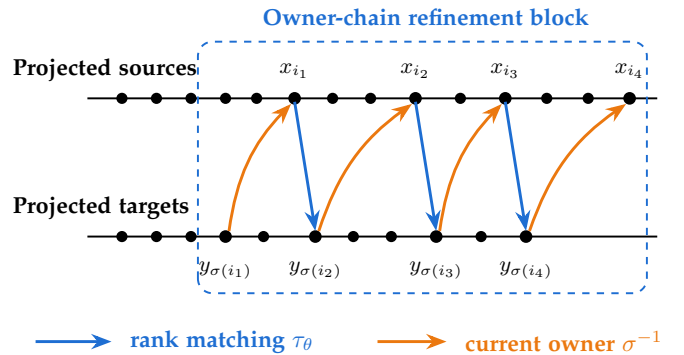
\begin{figure}[htbp]
\centering
\resizebox{\columnwidth}{!}{%
\begin{tikzpicture}[
    >=Stealth,
    every node/.style={font=\small}
]

\definecolor{rankblue}{RGB}{30,110,210}
\definecolor{ownerorange}{RGB}{235,120,15}

\node[
    anchor=west,
    font=\bfseries
]
at (0.15,2.42)
{Projected sources};

\node[
    anchor=west,
    font=\bfseries
]
at (0.15,0.42)
{Projected targets};

\begin{scope}[xshift=1.15cm]

\draw[thick] (0.20,2) -- (8.45,2);
\draw[thick] (0.20,0) -- (8.45,0);

\foreach \xx in {0.70,1.20,1.70}
{
    \fill (\xx,2) circle (2.3pt);
    \fill (\xx,0) circle (2.3pt);
}

\fill (2.20,2) circle (2.3pt);
\fill (2.65,2) circle (2.3pt);

\coordinate (x1) at (3.20,2);
\coordinate (x2) at (4.95,2);
\coordinate (x3) at (6.25,2);
\coordinate (x4) at (8.05,2);

\foreach \p in {x1,x2,x3,x4}
{
    \fill (\p) circle (2.6pt);
}

\node[above=5pt] at (x1) {$x_{i_1}$};
\node[above=5pt] at (x2) {$x_{i_2}$};
\node[above=5pt] at (x3) {$x_{i_3}$};
\node[above=5pt] at (x4) {$x_{i_4}$};

\fill (3.75,2) circle (2.3pt);
\fill (4.30,2) circle (2.3pt);

\fill (5.60,2) circle (2.3pt);

\fill (6.85,2) circle (2.3pt);
\fill (7.45,2) circle (2.3pt);

\coordinate (y1) at (2.20,0);

\fill (2.75,0) circle (2.3pt);

\coordinate (y2) at (3.50,0);
\coordinate (y3) at (5.25,0);
\coordinate (y4) at (6.55,0);

\foreach \p in {y1,y2,y3,y4}
{
    \fill (\p) circle (2.6pt);
}

\node[below=7pt] at (y1) {$y_{\sigma(i_1)}$};
\node[below=7pt] at (y2) {$y_{\sigma(i_2)}$};
\node[below=7pt] at (y3) {$y_{\sigma(i_3)}$};
\node[below=7pt] at (y4) {$y_{\sigma(i_4)}$};

\fill (4.05,0) circle (2.3pt);
\fill (4.60,0) circle (2.3pt);

\fill (5.90,0) circle (2.3pt);

\draw[
    dashed,
    thick,
    rankblue,
    rounded corners=6pt
]
(1.80,-0.82) rectangle (8.30,2.82);

\node[
    rankblue,
    font=\bfseries
]
at (5.10,3.18)
{Owner-chain refinement block};

\draw[
    ->,
    very thick,
    rankblue
]
($(x1)+(0,-0.05)$)
--
($(y2)+(0,0.08)$);

\draw[
    ->,
    very thick,
    rankblue
]
($(x2)+(0,-0.05)$)
--
($(y3)+(0,0.08)$);

\draw[
    ->,
    very thick,
    rankblue
]
($(x3)+(0,-0.05)$)
--
($(y4)+(0,0.08)$);

\draw[
    ->,
    very thick,
    ownerorange
]
($(y1)+(0.05,0.08)$)
to[bend left=18]
($(x1)+(-0.05,-0.08)$);

\draw[
    ->,
    very thick,
    ownerorange
]
($(y2)+(0.05,0.08)$)
to[bend left=18]
($(x2)+(-0.05,-0.08)$);

\draw[
    ->,
    very thick,
    ownerorange
]
($(y3)+(0.05,0.08)$)
to[bend left=18]
($(x3)+(-0.05,-0.08)$);

\draw[
    ->,
    very thick,
    ownerorange
]
($(y4)+(0.05,0.08)$)
to[bend left=18]
($(x4)+(-0.05,-0.08)$);

\draw[
    ->,
    very thick,
    rankblue
]
(-0.55,-1.52) -- (0.55,-1.52);

\node[
    anchor=west,
    rankblue,
    font=\bfseries
]
at (0.70,-1.52)
{rank matching $\tau_\theta$};

\draw[
    ->,
    very thick,
    ownerorange
]
(4.40,-1.52) -- (5.40,-1.52);

\node[
    anchor=west,
    ownerorange,
    font=\bfseries
]
at (5.55,-1.52)
{current owner $\sigma^{-1}$};

\end{scope}

\end{tikzpicture}
}
\caption{
Illustration of the owner-chain construction.
}
\label{fig:owner-chain-block}
\end{figure}
Specifically, for each owner-chain block involving $q$ source points
$(x_{i_1},\ldots,x_{i_q})$ and their associated target points
$(y_{\sigma(i_1)},\ldots,y_{\sigma(i_q)})$, we solve the corresponding local linear
assignment problem using LAP. If the optimal local reassignment decreases
the transport cost, it is accepted to update the current permutation.
Thus, the owner-chain construction provides a geometry-aware mechanism
for exploring higher-order local rearrangements without exhaustive
enumeration. Nevertheless, as demonstrated in Section~\ref{sec:synthetic-main},
the additional improvement obtained from this refinement does not compensate
for its computational overhead. In our experiments, standard pairwise
SWAP updates achieve better wall-clock efficiency, and therefore we use 
pairwise swaps as the default update strategy.
\subsection{Computational and memory complexity}
We briefly discuss the computational cost of the swap-pair search strategy.
Assume that each point cloud contains $N$ points in $\mathbb{R}^d$ and that
evaluating the cost $c(x,y)$ requires $O(d)$ operations. Hence,
testing whether a single candidate pair is a decreasing pair requires $O(d)$ operations.

For the sliced strategy, each sliced direction uses a one-dimensional projection to
generate structured candidate pairs. For a projection direction $\theta$,
computing
\[
s_i=\langle x_i,\theta\rangle,
\qquad
t_j=\langle y_j,\theta\rangle
\]
for all source and target points requires $O(Nd)$ operations. Sorting the two
sets of projected values costs $O(N\log N)$. After sorting, constructing the
temporary one-dimensional matching $\tau$, generating a set of
candidate pairs, and maintaining the inverse permutation $\sigma^{-1}$ require
only $O(N)$ additional operations. Testing the decreasing condition for all selected pairs requires at most $O(Nd)$ operations. Therefore, for $L$ sliced directions, the computational cost is
\[
O(L(Nd+N\log N)).
\]
Storing the two point clouds, projected coordinates, sorting indices, the current permutation and its inverse, and the candidate pairs requires $O(Nd)$ memory.
\section{Cycle Stability, Convergence, and Optimality}
In this section, we study the stability and convergence properties of SWAP and identify conditions under which pairwise stability implies global optimality.
\subsection{\(m\)-cycle stability}
The cyclic monotonicity theorem requires checking cycles of all possible
lengths. In practice, this is computationally infeasible for large-scale
problems. We therefore introduce a truncated version of cyclic monotonicity
by restricting the length of admissible decreasing cycles.

Let
\[
C_c(\sigma)=\sum_{i=1}^N c(x_i,y_{\sigma(i)})
\]
be the matching cost associated with a permutation
\(\sigma\in S_N\).

For distinct indices
\[
i_1,\ldots,i_m\in\{1,\ldots,N\},
\qquad 2\le m\le N,
\]
we denote the corresponding cycle by
\[
\gamma=(i_1,\ldots,i_m),
\]
with the convention
\[
i_{m+1}=i_1.
\]

The cyclic rearrangement induced by \(\gamma\), denoted by
\(\sigma^\gamma\), is defined by
\[
\sigma^\gamma(i_{r+1})=\sigma(i_r),
\qquad r=1,\ldots,m,
\]
and
\[
\sigma^\gamma(j)=\sigma(j),
\qquad
j\notin\{i_1,\ldots,i_m\}.
\]
Thus, \(\sigma^\gamma\) rotates the currently assigned targets among the
selected source points and leaves all other assignments unchanged.

The cost decrease induced by the cycle \(\gamma\) is
\[
\Delta_\gamma(\sigma)
=
\sum_{r=1}^m c(x_{i_{r+1}},y_{\sigma(i_r)})-\sum_{r=1}^m c(x_{i_r},y_{\sigma(i_r)}).
\]
If
\[
\Delta_\gamma(\sigma)<0,
\]
then \(\gamma\) is called a \emph{decreasing cycle}.

\begin{definition}[\(m\)-cycle stability]
Let \(2\le m\le N\). A permutation \(\sigma\in S_N\) is called
\(m\)-cycle stable if it contains no decreasing cycle of length at most \(m\).
\end{definition}

The case \(m=2\) corresponds to the absence of improving pairwise swaps.
Larger values of \(m\) rule out higher-order cyclic improvements.

\begin{theorem}[Monotonicity and global optimality of \(m\)-cycle stability]
Let \(\Omega_m\) denote the set of all \(m\)-cycle stable permutations.
Then the stability sets satisfy
\[
\Omega_N
\subseteq
\Omega_{N-1}
\subseteq
\cdots
\subseteq
\Omega_2 .
\]
Moreover,
\[
\Omega_N
=
\operatorname*{argmin}_{\sigma\in S_N}
\sum_{i=1}^N c(x_i,y_{\sigma(i)}) .
\]
In other words, \(N\)-cycle stability is equivalent to global optimality.
\end{theorem}

\begin{proof}
First, we prove the monotonicity. Let \(2\le \ell\le m\le N\), and suppose
that \(\sigma\in\Omega_m\). By definition, no decreasing cycle of
length at most \(m\) can decrease the cost. Since every cycle of length at
most \(\ell\) is also a cycle of length at most \(m\), no cycle of length at
most \(\ell\) can decrease the cost. Hence,
\[
\sigma\in\Omega_\ell.
\]
Therefore,
\[
\Omega_N
\subseteq
\Omega_{N-1}
\subseteq
\cdots
\subseteq
\Omega_2 .
\]

It remains to identify \(\Omega_N\). By definition, \(\sigma\in\Omega_N\)
if and only if for every \(k=2,\ldots,N\) and every collection of distinct
indices \(i_1,\ldots,i_k\), with \(i_{k+1}=i_1\), one has
\[
\sum_{r=1}^k c(x_{i_r},y_{\sigma(i_r)})
\le
\sum_{r=1}^k c(x_{i_{r+1}},y_{\sigma(i_r)}).
\]
This is exactly the cyclic monotonicity condition, which is equivalent to the
global optimality of \(\sigma\) by the cyclic monotonicity theorem. Therefore,
\[
\Omega_N
=
\operatorname*{argmin}_{\sigma\in S_N}
\sum_{i=1}^N c(x_i,y_{\sigma(i)}) .
\]
\end{proof}

The family of \(m\)-cycle stability conditions naturally motivates a local search strategy. In principle, global optimality requires ruling out decreasing cycles of all lengths up to \(N\), and searching over higher-order cycles can potentially yield lower-cost solutions. However, as demonstrated by our numerical experiments, higher-order owner-chain refinement introduces substantial computational overhead and reduces wall-clock efficiency. This tradeoff motivates the pairwise-swap strategy underlying SWAP: it provides a computationally inexpensive local search, whereas higher-order owner-chain refinement incurs substantially greater computational cost and is therefore less efficient in practice.

\subsection{Almost-Sure Convergence to Pairwise Stability 
}
In this section, we show that, under a mild nonparallel condition, the SWAP method reaches $\Omega_2$ almost surely after finitely many iterations. The key step is to prove that any prescribed improving pair has a strictly positive probability of being generated as a candidate pair by a random one-dimensional projection.

\begin{theorem}[Positive-probability candidate pair proposal]\label{thm:parallel}
Let
\[
X=\{x_1,\ldots,x_N\}\subset\mathbb R^d,
\;
Z=\{z_1,\ldots,z_N\}\subset\mathbb R^d,
\;
d\geq2,
\]
where
\[
z_j=y_{\sigma(j)}
\]
is the target currently owned by $x_j$. Fix two distinct indices
$i$ and $j$. Assume that
\[
x_i-x_k
\not\parallel
x_i-x_{k'}
\qquad
\text{for all distinct }k,k'\neq i,
\tag{NP1}
\]
\[
z_j-z_\ell
\not\parallel
z_j-z_{\ell'}
\qquad
\text{for all distinct }\ell,\ell'\neq j,
\tag{NP2}
\]
and
\[
x_i-x_k
\not\parallel
z_j-z_\ell
\qquad
\text{for all }k\neq i,\ \ell\neq j.
\tag{NP3}
\]
We refer to \textup{(NP1)--(NP3)} collectively as the
\emph{nonparallel condition}.

Let $\Theta$ be uniformly distributed on $\mathbb S^{d-1}$.
A projection tie is said to occur in a direction $\theta$ if two
distinct source points or two distinct target points have the same
projected value, i.e.,
\[
\langle\theta,x_k\rangle
=
\langle\theta,x_{k'}\rangle
\]
for some $k\neq k'$, or
\[
\langle\theta,z_\ell\rangle
=
\langle\theta,z_{\ell'}\rangle
\]
for some $\ell\neq\ell'$.

For a direction $\theta$ without projection ties, define
\[
r_X(i;\theta)
=
1+
\#\left\{
k:
\langle\theta,x_k\rangle
<
\langle\theta,x_i\rangle
\right\},
\]
and
\[
r_Z(j;\theta)
=
1+
\#\left\{
\ell:
\langle\theta,z_\ell\rangle
<
\langle\theta,z_j\rangle
\right\}.
\]
Then
\[
\mathbb P_{\Theta}
\left(
r_X(i;\Theta)=r_Z(j;\Theta)
\right)>0.
\]
Consequently, the pair $(i,j)$ is generated as a  candidate pair
with strictly positive probability. This result is independent of the
choice of transport cost and therefore holds, in particular, for any
distance-based cost.
\end{theorem}

The proof uses a continuous-path argument together with the fact that the rank
difference changes only by unit steps; details are provided in Appendix A.1.

\begin{remark}
The nonparallel conditions require only distinct directions, not linear
independence, and therefore allow $N\gg d$. Moreover, for independently
sampled point clouds from absolutely continuous distributions, configurations
violating these conditions form a measure-zero set.
\end{remark}

In particular, if $(i,j)$ is a decreasing pair under the current
permutation, the previous theorem shows that it has a strictly positive
probability of being generated as a candidate pair. This property implies
that repeated independent sliced searches cannot remain indefinitely at a
permutation that still admits an improving swap.

\begin{theorem}[Almost-sure convergence to $\Omega_2$]
Assume that, at every iterate $\sigma\notin \Omega_2$, there exists at least one decreasing pair $(i,j)$ for which the nonparallel conditions of Theorem ~\ref{thm:parallel} hold. Suppose that independent random projection
directions are sampled at successive iterations and that an improving update is
performed whenever a decreasing pair is identified. Then SWAP
reaches $\Omega_2$ almost surely after finitely many iterations.
\end{theorem}

\begin{proof}
Suppose that the current permutation satisfies
\[
\sigma\notin\Omega_2.
\]
Then there exists at least one decreasing pair $(i,j)$. By the previous
theorem, this pair is generated as a candidate pair with some strictly
positive probability. Therefore, as long as the current permutation remains
unchanged, repeated independent random projections identify an improving swap
pair with probability one.

Every accepted swap strictly decreases the transport cost. Hence the algorithm
cannot revisit a previously visited permutation. Since $S_N$ contains only
finitely many permutations, only finitely many strictly decreasing updates are
possible.

If the algorithm never reached $\Omega_2$, it would therefore eventually
remain forever at some fixed permutation outside $\Omega_2$. However, the
probability of remaining forever at any such permutation is zero by the
positive-probability proposal result. Consequently, the sliced search reaches
$\Omega_2$ almost surely after finitely many iterations.
\end{proof}

Therefore, repeated random sliced searches almost surely drive the permutation
to $\Omega_2$, where no further decreasing pair exists. 
\subsection{Sufficient conditions for global optimality}

In the previous section, we proved that, under a mild nonparallel condition, the SWAP method reaches \(\Omega_2\) almost surely after finitely many iterations. In general, however,
\[
\Omega_N \subseteq \Omega_2,
\]
so pairwise stability does not necessarily imply global optimality. In particular, a permutation may be stable with respect to all pairwise exchanges while still admitting a cost-decreasing cycle involving three or more elements.

In one dimension, such a situation cannot occur due to the ordering
structure induced by Monge costs. We first recall this classical case.

\begin{proposition}
Assume that \(x_i,y_i\in\mathbb R\) and
\[
    c(x,y)=h(|x-y|)
\]
satisfies the strict Monge condition
\[
c(x_1,y_1)+c(x_2,y_2)
<
c(x_1,y_2)+c(x_2,y_1),
\]
whenever \(x_1<x_2\) and \(y_1<y_2\). In particular, this includes $h(r)=r^p$ with $p>1$.

If a permutation \(\sigma\in S_N\) contains no improving pairwise
swap, then \(\sigma\) is globally optimal for
\[
C_c(\sigma)=\sum_{i=1}^N c(x_i,y_{\sigma(i)}).
\]
Consequently,
\[
    \Omega_2=\Omega_N .
\]
\end{proposition}

\begin{proof}
Suppose that the matching is not order preserving. Then there exist \(i,j\)
such that
\[
x_i<x_j,
\qquad
y_{\sigma(i)}>y_{\sigma(j)}.
\]
Applying the strict Monge inequality to the ordered pairs
\(x_i<x_j\) and \(y_{\sigma(j)}<y_{\sigma(i)}\) gives
\[
c(x_i,y_{\sigma(j)})+c(x_j,y_{\sigma(i)})
<
c(x_i,y_{\sigma(i)})+c(x_j,y_{\sigma(j)}).
\]
Thus, \((i,j)\) is an improving swap, contradicting the assumption. Hence the
matching is order preserving. By the one-dimensional monotone rearrangement
theorem, every order-preserving matching is globally optimal.
\end{proof}

The above argument relies essentially on the existence of a global ordering. In higher dimensions, no such ordering exists in general, and 2-cycle stability is no longer sufficient for global
optimality. The cyclic monotonicity characterization of optimal transport shows that higher-order cycles must also be controlled.

We next present two sufficient conditions under which pairwise stability guarantees global optimality. More precisely, if a permutation \(\sigma\in\Omega_2\) satisfies either of these conditions, then \(\sigma\in\Omega_N\), and hence \(\sigma\) is globally optimal.
The first is the classical Monge array condition, while the second is
a new exchange-circulation condition tailored to the swap framework.

\begin{proposition}
    Let
\[
    C_{ij}:=c(x_i,y_j),
    \qquad i,j\in\{1,\ldots,N\},
\]
and suppose that the source and target indices are ordered such that
the cost matrix \(C\) satisfies the Monge condition
\begin{equation}\label{eq:monge-cost-matrix}
    C_{ij}+C_{k\ell}
    \leq
    C_{i\ell}+C_{kj},
    \qquad
    i<k,\quad j<\ell.
\end{equation}
Then:
\[
    \Omega_2=\Omega_N.
\]
\end{proposition}

\begin{theorem}[Anchored exchange-circulation condition]
\label{thm:anchored-tdec}
Let $\sigma\in S_N$ be a permutation, and define the directed exchange
increment from $i$ to $j$ by
\begin{equation}
w_{ij}(\sigma)
=
c(x_j,y_{\sigma(i)})
-
c(x_i,y_{\sigma(i)}).
\end{equation}
Its symmetric and antisymmetric parts are given by
\begin{equation}
s_{ij}(\sigma)
=
\frac{w_{ij}(\sigma)+w_{ji}(\sigma)}{2},
\qquad
a_{ij}(\sigma)
=
\frac{w_{ij}(\sigma)-w_{ji}(\sigma)}{2}.
\end{equation}
Thus,
\[
w_{ij}(\sigma)=s_{ij}(\sigma)+a_{ij}(\sigma),
\]
where $s_{ij}=s_{ji}$ and $a_{ij}=-a_{ji}$.

Assume that $\sigma$ is 2-cycle stable, i.e.,
$\sigma\in\Omega_2$. Fix an index $o\in\{1,\ldots,N\}$ and define
the anchored triangle circulation
\begin{equation}
\kappa_{oij}(\sigma)
=
a_{oi}(\sigma)
+
a_{ij}(\sigma)
+
a_{jo}(\sigma).
\end{equation}
Suppose that there exists an anchor $o$ such that
\begin{equation}
\label{eq:anchored-tdec}
\kappa_{oij}(\sigma)
\ge
-s_{ij}(\sigma),
\qquad
\text{for all } i,j\in\{1,\ldots,N\}.
\end{equation}
Then $\sigma$ is globally optimal. Equivalently,
\[
\sigma\in\Omega_N.
\]
\end{theorem}

The proofs of Proposition 3 and Theorem 5 are provided in Appendices A.2 and A.3, respectively.

Although the sufficient conditions above need not hold in general, the numerical experiments in the next section show that SWAP nevertheless attains near-optimal objective values on the controlled benchmarks, often reaching the certified optimum to numerical precision.

\section{Numerical Experiments}\label{sec:experiments}
We benchmark SWAP in three stages. Section~\ref{sec:imagenet} tests whether the method can solve a genuinely large, high-dimensional assignment problem and illustrates both PCA continuation and the anytime feasibility of its iterates. Section~\ref{sec:synthetic-main} then isolates four solver-level properties under controlled protocols: accuracy against a known optimum, scaling across problem sizes and dimensions, runtime robustness across data geometries, and the efficiency of the default pairwise update relative to higher-order refinement. Finally, Section~\ref{sec:merfish-main} evaluates whether the resulting permutation is useful for downstream spatial alignment, using gene-expression measurements withheld from the solver as an independent evaluation signal.

Unless otherwise stated, SWAP is initialized by a one-dimensional sliced matching. A random direction $\theta\in\mathbb{S}^{d-1}$ is sampled, the source and target points are independently sorted by their projected coordinates, and points of equal projected rank are matched to form the initial permutation. All SWAP results use fixed direction budgets or fixed wall-clock budget, specified separately for each experiment. The competing methods are run using the default configurations of their public implementations.

The evaluated methods differ in the type of object they return. Exact assignment, minibatch assignment, HiRef, and SWAP return permutation-valued solutions, whereas Sinkhorn and HALO return transport couplings. Throughout the experiments, ``Cost'' denotes the original unregularized transport objective under the stated ground cost. For a permutation $\sigma$, this is the mean assignment cost $C(\sigma)=\frac{1}{N}\sum_{i=1}^{N}c(x_i,y_{\sigma(i)}).$
For a coupling $P$, it is the corresponding primal cost $C(P)=\sum_{i,j}P_{ij}c(x_i,y_j).$

All experiments are run on an NVIDIA GeForce RTX 5090. A run is declared time-out (TO) if its wall-clock time exceeds 18,000 seconds and out-of-memory (OOM) if the peak memory consumption exceeds 32 GiB. 
\label{sec:numerical-experiments}

\subsection{Large-Scale ImageNet Assignment}
\label{sec:imagenet}

We begin our empirical study with a large-scale assignment problem based on the ImageNet/ILSVRC training set~\cite{dengImageNetLargeScale2009,russakovskyImageNetLargeScale2015}. Using an ImageNet-pretrained ResNet-50~\cite{heDeepResidualLearning2016}, we extract 2,048-dimensional features from 1,281,000 images and split them evenly into source and target point clouds, each containing $N=$ 640,500 points in dimension $d=$ 2,048. We then seek a bijection minimizing the mean squared Euclidean assignment cost.

\noindent\textbf{Baseline comparison.}
Standard full-matrix solvers such as JV~\cite{jonkerShortestAugmentingPath1987} and Sinkhorn~\cite{cuturiSinkhornDistancesLightspeed2013} are impractical at this scale under our memory budget. We therefore compare SWAP with two large-scale OT baselines, HiRef~\cite{halmosHierarchicalRefinementOptimal2025} and HALO~\cite{xiaMemoryEfficientHierarchical2026}. We additionally report minibatch OT (MB)~\cite{fatrasLearningMinibatchWasserstein2020}, a method commonly used in machine-learning applications. We divide the source and target point clouds into paired minibatches of size $B\in\{128,512,1024\}$ and solve each minibatch assignment exactly. SWAP performs 20,000 directions in the original feature space. 

\begin{table}[!h]
	\caption{Large-scale ImageNet assignment results.}
	\label{tab:imagenet}
	\centering
	\scriptsize
	\setlength{\tabcolsep}{1.8pt}
	\begin{tabular}{@{}lcccc@{}}
		\toprule
		Method
		& Cost
		& Same class
		& Time (s)
		& Mem. (GiB) \\
		\midrule
		SWAP              & 300.596          & 35.77\%          & 128.4          & \textbf{10.87} \\
		SWAP (PCA)     & \textbf{267.034} & \textbf{40.78\%} & \textbf{125.8} & 11.48 \\
		HiRef             & 341.336          & 27.63\%          & 1777.5         & 29.55 \\
		\midrule
		MB ($B=128$)      & 486.158          & 5.04\%          & 4.8    & 9.81 \\
		MB ($B=512$)      & 420.433          & 11.93\%         & 20.4   & 9.82 \\
		MB ($B=1024$)     & 390.064          & 17.00\%         & 40.8   & 9.84 \\
		\bottomrule
	\end{tabular}
\end{table}

As shown in Table~\ref{tab:imagenet}, SWAP attains a lower quadratic cost than every completed baseline. Compared with HiRef, it also requires substantially less time and memory. HALO is omitted from the table because it times out before completing the instance. These results show that SWAP remains practical in the 2,048-dimensional feature space.

Class labels are not used by any method for optimization or endpoint selection. We also report the same-class rate, defined as the fraction of matched image pairs sharing the same ImageNet class label. The same-class results further show that the lower assignment costs achieved by SWAP are accompanied by stronger semantic agreement in the pretrained feature space. These results suggest that SWAP may be useful in deep-learning and image-processing applications that require large-scale one-to-one matching between visual representations.

\noindent\textbf{PCA continuation.}
Since SWAP can resume from any feasible permutation, an assignment obtained in a low-dimensional space can be used directly to warm-start optimization in the original space. We first perform PCA on the combined source and target features and project both point clouds onto the leading 128 components. We run SWAP for 100,000 directions on these 128-dimensional representations and then use the resulting permutation to initialize 8,000 further directions on the original 2,048-dimensional features. The PCA timings in Table~\ref{tab:imagenet} and Fig.~\ref{fig:imagenet-convergence} exclude the 17.48-s preprocessing step.

As shown in Fig.~\ref{fig:imagenet-convergence}, most of the reduction in the original-space assignment cost already occurs during the reduced-space stage before the marked handoff, while the full-dimensional stage further refines the resulting permutation. This behavior illustrates how SWAP can perform most of the search in a less expensive representation and then continue optimizing the resulting feasible permutation under the original objective. The trajectory also demonstrates the anytime property as every intermediate SWAP state can be returned immediately as a feasible one-to-one assignment. Since the trajectory is still decreasing at the fixed-budget endpoint, the reported schedule is intended to illustrate the benefit of PCA continuation rather than the best attainable performance of this strategy.

\begin{figure}[t]
	\centering
	\includegraphics[width=\columnwidth]{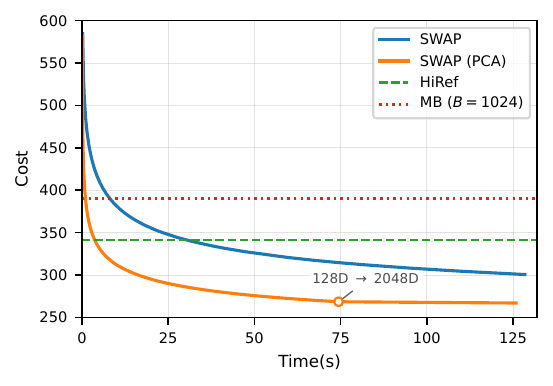}
	\caption{Assignment cost on the ImageNet instance as a function of solver-core time. All displayed costs are evaluated in the original feature space, and the marker indicates the handoff from 128-dimensional PCA to full-dimensional refinement. The dashed and dotted horizontal lines indicate the final costs of HiRef and MB ($B=1024$), respectively.}
	\label{fig:imagenet-convergence}
\end{figure}

\noindent\textbf{Cosine cost.}
Because cosine similarity is widely used in computer vision~\cite{radfordLearningTransferableVisual2021,radenovicFineTuningCNNImage2018,wangNormFaceL2Hypersphere2017}, we also compare the methods on ImageNet using the cosine cost
\[
    c(x,y)=1-\frac{\langle x,y\rangle}{\lVert x\rVert_2\lVert y\rVert_2}.
\]
As shown in Table~\ref{tab:imagenet-cosine}, SWAP again attains the lowest mean cosine cost among the compared methods, while requiring substantially less time and memory than HiRef.

\begin{table}[!h]
	\caption{Cosine-cost ImageNet assignment results.}
	\label{tab:imagenet-cosine}
	\centering
	\scriptsize
	\setlength{\tabcolsep}{1.8pt}
	\begin{tabular}{@{}lcccc@{}}
		\toprule
		Method
		& Cosine cost
		& Same class
		& Time (s)
		& Mem. (GiB) \\
		\midrule
		SWAP       & \textbf{0.1774} & \textbf{36.30\%} & \textbf{131.6} & \textbf{10.90} \\
		HiRef          & 0.1843          & 33.70\%          & 968.6          & 29.55 \\
		\midrule
		MB ($B=128$)  & 0.2881          & 5.12\%          & 3.8   & 9.81 \\
		MB ($B=512$)  & 0.2485          & 12.10\%         & 14.0  & 9.82 \\
		MB ($B=1024$) & 0.2302          & 17.24\%         & 25.3  & 9.84 \\
		\bottomrule
	\end{tabular}
\end{table}

\subsection{Accuracy, Scalability, Robustness, and Efficiency}
\label{sec:synthetic-main}
We consider equally weighted source and target point clouds with the same cardinality and use the squared Euclidean ground cost throughout the synthetic experiments. All methods within a given configuration receive identical source and target samples. We use seed $200$ throughout, and the different problem sizes for a fixed dataset and dimension are nested prefixes of the same master instance. 

\noindent\textbf{Accuracy against the known optimum.}
The first synthetic experiment evaluates how close the permutation returned by SWAP is to a ground-truth optimum. We construct the target point cloud through the gradient of a strongly convex potential. Let
\begin{equation*}
    \phi_d(x)
    =
    \frac{\alpha}{2}\|x\|_2^2
    +
    \frac{\eta}{2}
    \sum_{j=1}^{d-1}(x_{j+1}-x_j)^2
    +
    \beta
    \sum_{j=1}^{d}\log\cosh(x_j),
    \label{eq:planted-potential-main}
\end{equation*}
with
\[
    \alpha=0.80,\qquad
    \eta=0.15,\qquad
    \beta=0.35.
\]
Writing $L_d$ for the path-graph Laplacian, the associated map is
\begin{equation*}
    T^\star_d(x)
    =
    \nabla\phi_d(x)
    =
    \alpha x+\eta L_dx+\beta\tanh(x).
    \label{eq:planted-map-main}
\end{equation*}
We sample
\begin{equation*}
    x_i\sim\mathcal{N}(0,I_{64}),
    \qquad
    y_i^\star=T^\star_{64}(x_i).
\end{equation*}
By Theorem 2.12 in \cite{villani2003topics}, the pairs $(x_i,y_i^\star)$ form an optimal assignment. To construct the actual solver input, we randomly permute the target samples $\{y_i^\star\}_{i=1}^N$, and then initialize each solver on the resulting source--target point clouds according to the common initialization protocol.

We test $N\in\{2^{16},2^{18},2^{20}\}$ in dimension $d=64$ and compare SWAP with HiRef and HALO. In addition to
the transport cost, we report the relative gap to the ground-truth optimum,
\begin{equation}\label{eq:gap}
    \operatorname{Gap}(\sigma)
    =
    \frac{C(\sigma)-C(\sigma^\star)}
         {C(\sigma^\star)} .
\end{equation}
SWAP is run for a fixed budget of 1,000,000 directions.

\begin{table}[t]
\centering
\caption{64D Brenier-map ground-truth assignment results. $-$ denotes a configuration not attempted after the same method had already timed out at a smaller scale.}
\label{tab:planted-gaussian64d}
\scriptsize
\setlength{\tabcolsep}{3.4pt}

\begin{tabular}{llrrr}
\toprule
 & Method & $N=2^{16}$ & $N=2^{18}$ & $N=2^{20}$ \\
\midrule

\multirow{3}{*}{cost}
 & HALO
 & 41.149 & \TO & -- \\

 & HiRef
 & 109.086 & 110.719 & 112.619 \\

 & SWAP
 & \textbf{9.142} & \textbf{9.138} & \textbf{9.178} \\

\addlinespace

\multirow{3}{*}{gap (\%)}
 & HALO
 & 350.106 & \TO & -- \\

 & HiRef
 & 1093.244 & 1111.620 & 1132.452 \\

 & SWAP
 & \textbf{0} & \textbf{0} & \textbf{0.4} \\

\addlinespace

\multirow{3}{*}{time (s)}
 & HALO
 & 757.358 & \TO & -- \\

 & HiRef
 & 433.477 & 1724.218 & 6907.858 \\

 & SWAP
 & \textbf{39.665} & \textbf{159.423} & \textbf{650.448} \\

\addlinespace

\multirow{3}{*}{Mem. (GiB)}
 & HALO
 & 2.234 & \TO & -- \\

 & HiRef
 & \textbf{0.902} & 1.469 & 3.771 \\

 & SWAP
 & {1.439} & \textbf{1.441} & \textbf{2.395} \\

\bottomrule
\end{tabular}

\end{table}

Table~\ref{tab:planted-gaussian64d} shows that SWAP consistently achieves substantially smaller objective gaps than the hierarchical baselines across all tested scales. At the smaller sizes, the fixed direction budget is relatively generous, allowing SWAP to approach the optimum closely. As the problem size increases, the same fixed budget becomes progressively more restrictive, yet SWAP continues to maintain a low optimality gap. With the dimension and direction budget fixed, the observed runtime grows approximately linearly with \(N\) over the tested range, consistent with the $O(Nd+NlogN)$ per-direction complexity of SWAP, while memory usage remains moderate as the problem size increases. HiRef also exhibits approximately linear runtime growth over the tested range, while its absolute runtime is about an order of magnitude larger than that of SWAP, and the cost is consistently higher. HALO ceases to complete the larger instances. 

\noindent\textbf{Scaling across size and dimension.}
Next, we test SWAP with a checkerboard benchmark that varies both problem size and ambient dimension and includes a broader range of OT solvers. Let $K=4$ and partition $[-2,2]^d$ into $K^d$ Cartesian
cells. For $\ell=0,\ldots,K-1$, define
\[
I_\ell =
\begin{cases}
[-2+4\ell/K,\,-2+4(\ell+1)/K), & \ell=0,\ldots,K-2,\\
[-2+4(K-1)/K,\,2], & \ell=K-1,
\end{cases}
\]
and, for $\mathbf{k}=(k_1,\ldots,k_d)\in\{0,\ldots,K-1\}^d$,
\[
C_{\mathbf{k}} = I_{k_1}\times\cdots\times I_{k_d}.
\]
We assign each cell a parity according to
\[
p(\mathbf{k}) = \sum_{j=1}^d k_j \pmod 2.
\]
Source samples are drawn uniformly from cells with
$p(\mathbf{k})=0$, while target samples are drawn uniformly from
cells with $p(\mathbf{k})=1$. We test $d\in\{2,16,64\}, N\in\{2^{13},2^{20}\}.$ Exact and dense Sinkhorn are included at $N=2^{13}$, while HALO, HiRef, and SWAP are evaluated over the
full scalable grid. SWAP is run for $200,000$ directions in every checkerboard configuration.

\begin{table*}[t]
\centering
\caption{Checkerboard assignment results. When the exact solver is available, boldface in the Cost rows indicates the best result among the non-exact methods.}
\label{tab:checkerboard-highdim}
\scriptsize
\setlength{\tabcolsep}{3.6pt}
\renewcommand{\arraystretch}{1.10}

\resizebox{0.98\textwidth}{!}{%
\begin{tabular}{c|l|rrrrr|rrr}
\toprule

\multirow{2}{*}{}
& \multirow{2}{*}{}
& \multicolumn{5}{c|}{$N=2^{13}$}
& \multicolumn{3}{c}{$N=2^{20}$} \\
\cmidrule(lr){3-7}
\cmidrule(lr){8-10}

&
& Exact & Sinkhorn & HALO & HiRef & SWAP
& HALO & HiRef & SWAP
\\
\midrule

\multirow{3}{*}{$d=2$}

& Cost
&0.279&0.317&\textbf{0.279}&0.343&0.282
&\textbf{0.267}&0.332&0.270
\\

& Time (s)
&49.760&4.475&{3.529}&48.149&\textbf{2.964}
&452.942&3836.114&\textbf{71.085}
\\

& Mem. (GiB)
&1.073&3.141&\textbf{0.643}&0.723&{0.787}
&3.914&2.152&\textbf{1.578}
\\

\midrule

\multirow{3}{*}{$d=16$}

& Cost
&9.252&\textbf{9.406}&10.138&11.798&10.948
&\TO&7.791&\textbf{6.740}
\\

& Time (s)
&{4.362}&4.482&20.563&55.905&\textbf{3.396}
&\TO&6842.162&\textbf{89.855}
\\

& Mem. (GiB)
&1.076&3.143&0.783&\textbf{0.723}&{0.803}
&\TO&2.488&\textbf{1.809}
\\

\midrule

\multirow{3}{*}{$d=64$}

& Cost
&94.974&\textbf{95.763}&111.067&106.629&103.028
&\TO&100.370&\textbf{93.615}
\\

& Time (s)
&5.612&{4.478}&23.188&56.042&\textbf{3.990}
&\TO&7017.423&\textbf{132.412}
\\

& Mem. (GiB)
&1.093&3.141&0.795&\textbf{0.740}&{0.830}
&\TO&3.771&\textbf{2.203}
\\

\bottomrule
\end{tabular}
}
\end{table*}

Table~\ref{tab:checkerboard-highdim} shows that SWAP remains competitive across a substantially broader range of dimensions and problem sizes. For most small-scale configurations, the exact solver and Sinkhorn attain lower transport costs, while SWAP consistently uses the least time across all tested dimensions. SWAP nevertheless achieves competitive objective values with relatively low memory usage, and its competitiveness compared to the hierarchical methods improves as the ambient dimension increases. At the large scale, the advantages of SWAP become more pronounced, particularly as the ambient dimension increases. In the higher-dimensional instances, SWAP achieves lower transport costs while using less memory than HiRef, and its computational advantage is also substantial: HALO does not complete within the computational budget, while HiRef requires more than fifty times the runtime of SWAP. Across problem sizes and dimensions, SWAP handles the increase in both $N$ and $d$ with comparatively mild growth in runtime and memory. Moreover, we emphasize that Sinkhorn and HALO return coupling-valued solutions, whereas SWAP maintains an exactly feasible permutation throughout optimization.

\noindent\textbf{Runtime robustness across data geometries.}
We further examine whether the runtime behavior remains stable
across diverse three-dimensional data geometries. We consider 12 fixed
ordered geometry pairs from ModelNet40~\cite{wu20153dshapenets}. 
Fig.~\ref{fig:cross-geometry-runtime} reports the median runtime across the
geometry pairs, with the shaded region showing the
corresponding minimum--maximum range.

\begin{figure}[t]
    \centering
    \includegraphics[width=\columnwidth]
    {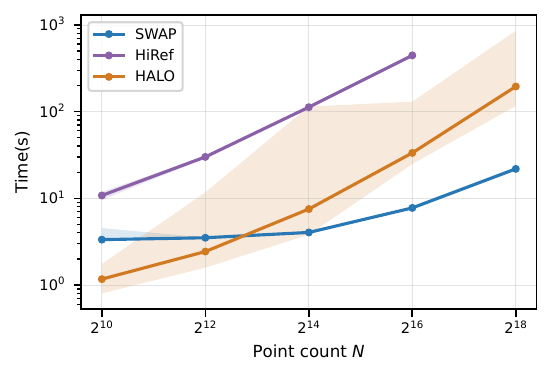}
    \caption{Runtime scaling across 12 fixed ordered ModelNet40 geometry pairs,
including a same-mesh independently sampled control, within-category shape
variations, and substantially different cross-category pairs. SWAP and HALO are evaluated at all scales, HiRef up to \(2^{16}\). At \(N=2^{16}\) and \(N=2^{18}\), HALO encounters out-of-memory failures on 1 and 2 of the 12 geometry pairs, respectively, and the corresponding failed runs are excluded from the plotted statistics.}
    \label{fig:cross-geometry-runtime}
\end{figure}

Figure~\ref{fig:cross-geometry-runtime} shows that SWAP maintains a consistently
narrow runtime range across the tested problem sizes, while exhibiting the same scaling with problem size. The small
variation across substantially different three-dimensional shape pairs indicates
that the runtime of SWAP is largely insensitive to the tested data geometry.
This is consistent with its fixed-budget structure: the dominant computational
workload depends primarily on \(N\), \(d\), and the proposal budget rather than
on the particular geometry of the dataset.

\noindent\textbf{Pairwise Swaps versus Higher-Order Refinement.}
The Owner-Chain construction illustrated in
Fig.~\ref{fig:owner-chain-block} extends the
pairwise SWAP update to higher-order local rearrangements. We examine whether
the additional descent obtained from this refinement justifies its
computational overhead. On the \(d=16\), \(N=2^{20}\) instance of the Brenier-map ground-truth
benchmark introduced earlier in this section, we compare plain SWAP with
variants that insert an Owner-Chain refinement of maximum group size \(m=256\)
every \(r\in\{2\mathrm{k},4\mathrm{k},8\mathrm{k},16\mathrm{k},
32\mathrm{k},64\mathrm{k},128\mathrm{k},256\mathrm{k}\}\) ordinary SWAP
sliced directions, under comparable wall-clock budgets of approximately
\(1,500\,\mathrm{s}\).

\begin{figure}[t]
    \centering
    \includegraphics[width=\columnwidth]
    {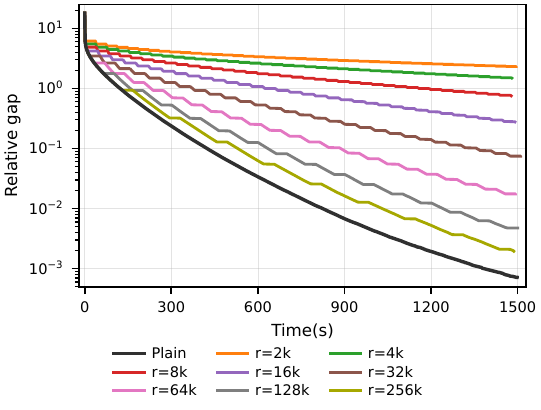}
   \caption{Relative optimality gap, as defined in Eq.~\eqref{eq:gap},
versus wall-clock time on the Brenier-map ground-truth benchmark with
\(d=16\) and \(N=2^{20}\). The eight refinement variants use \(m=256\),
with \(r\) ordinary SWAP sliced directions between successive Owner-Chain
refinement calls.}
    \label{fig:refinement-runtime}
\end{figure}

Figure~\ref{fig:refinement-runtime} shows that all variants substantially
reduce the objective, and higher-order refinement can produce additional local
improvements. These improvements, however, do not compensate for the cost of
the refinement calls: more frequent refinement leads to slower progress in
wall-clock time, while plain SWAP achieves the best efficiency. This supports
pairwise updates as the default choice, with higher-order refinement retained
as an optional but comparatively expensive enhancement.

\subsection{Spatial Alignment of MERFISH Brain Sections}
\label{sec:merfish-main}

We evaluate SWAP on the MERFISH mouse-brain alignment problem considered in HiRef~\cite{halmosHierarchicalRefinementOptimal2025}. The source and target point clouds are Slice~2 Replicate~3 and Slice~2 Replicate~2, respectively, from the Vizgen MERFISH Mouse Brain Receptor Map~\cite{chen2015merfish,vizgenMouseBrainMap}. Following the same preprocessing, the two point clouds are centered independently, the target is rotated counterclockwise by $45^\circ$, and the source is uniformly subsampled so that both point clouds contain \(N=84{,}172\) points. The matching is computed using only the two-dimensional spatial coordinates, with squared Euclidean cost. We compare SWAP, run for 2,000 directions, with HiRef and MB~\cite{fatrasLearningMinibatchWasserstein2020} using batch sizes $B\in\{128,256,512\}$.

Gene-expression measurements are not provided to the solvers and are used only to evaluate the resulting spatial matching. For each gene, the source expression values are transferred to the matched target locations and compared with the observed target values after averaging both fields over $200\mu\mathrm{m}$ square bins. We discard bins containing fewer than five target points and retain 379 genes after excluding blank probes and genes detected in fewer than $1\%$ of the points in either slice. The evaluation score is the median, over these genes, of the cosine similarity between the transferred and observed fields. It therefore quantifies the extent to which a matching computed solely from spatial coordinates preserves gene-expression structure that is not used during matching. This protocol follows previous spatial-alignment studies~\cite{clifton2023stalign,halmosHierarchicalRefinementOptimal2025}.

\begin{figure*}[t]
	\centering
	\includegraphics[width=\textwidth]{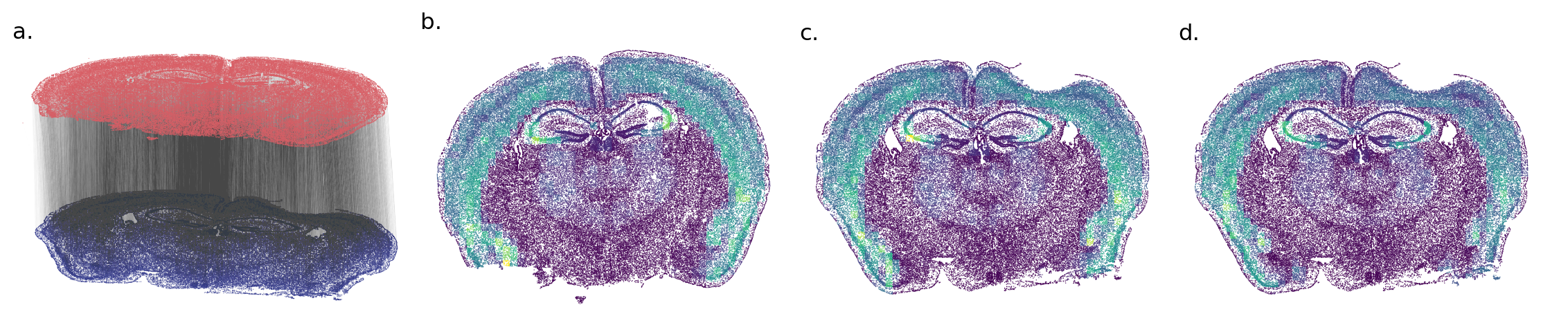}
	\caption{Spatial alignment of two MERFISH mouse-brain sections using SWAP. (a) One-to-one correspondences between the source and target cells. (b) Gene \textit{Slc17a7} abundance in the source section. (c) Source abundance transferred to the target locations through the SWAP assignment. (d) Observed abundance in the target section. Transferred abundances have cosine similarity 0.8858 with the observed abundances in the target.}
	\label{fig:merfish-spatial}
\end{figure*}

\begin{table}[t]
	\caption{MERFISH alignment results. The reported score is the median cosine similarity over 379 retained genes.}
	\label{tab:merfish-results}
	\centering
	\setlength{\tabcolsep}{4.5pt}
	\begin{tabular}{lccc}
		\toprule
		Method & Time (s) & Median cosine & Memory (MiB) \\
		\midrule
		SWAP      & \textbf{3.368}    & \textbf{0.7267} & \textbf{138.2} \\
		HiRef        & 1805.145 & 0.7249 & 988.0 \\
		\midrule
		MB ($B=128$) & 4.464    & 0.6652 & 9.6 \\
		MB ($B=256$) & 5.046    & 0.6921 & 10.2 \\
		MB ($B=512$) & 6.911    & 0.7090 & 12.4 \\
		\bottomrule
	\end{tabular}
\end{table}

Table~\ref{tab:merfish-results} shows that SWAP attains the highest gene-expression similarity among the tested methods while requiring only a small computational cost. It achieves alignment quality comparable to HiRef but is approximately 536x faster, and it consistently outperforms MB in alignment quality while remaining computationally fast. Figure~\ref{fig:merfish-spatial} provides a qualitative example for gene Slc17a7, for which
the abundance transferred by SWAP reproduces the
principal spatial pattern observed in the target section.

\section{Conclusion}
In this work, we proposed SWAP, an iterative permutation-based
framework for high-dimensional and large-scale discrete Monge optimal transport between equally weighted point clouds.
The method updates a feasible transport permutation through
cost-decreasing swap operations and uses random one-dimensional
projections to identify candidate pairs efficiently. We established a
positive-probability guarantee for the sliced proposal mechanism and
showed that repeated independent sliced searches reach pairwise
stability almost surely under mild nonparallel conditions. We further
derived sufficient conditions under which pairwise stability implies
global optimality. Numerical experiments demonstrated favorable
accuracy, scalability, and computational efficiency on large-scale and
high-dimensional problems. While pairwise stability does not guarantee
global optimality in general, the proposed framework provides an
efficient anytime approach for directly optimizing permutation-valued
transport assignments at scales beyond conventional dense assignment methods.

\clearpage
\section*{Appendices}
\setcounter{equation}{0}
\renewcommand{\theequation}{S\arabic{equation}}
\setcounter{figure}{0}
\renewcommand{\thefigure}{S\arabic{figure}}
\setcounter{table}{0}
\renewcommand{\thetable}{S\arabic{table}}
\setcounter{theorem}{0}
\renewcommand{\thetheorem}{S\arabic{theorem}}
\setcounter{proposition}{0}
\renewcommand{\theproposition}{S\arabic{proposition}}
\setcounter{lemma}{0}
\renewcommand{\thelemma}{S\arabic{lemma}}
\setcounter{corollary}{0}
\renewcommand{\thecorollary}{S\arabic{corollary}}

\noindent
This supplementary material provides complete proofs and additional numerical studies that complement the main paper. Appendix A presents detailed proofs of the principal theoretical results. Appendix B provides geometric visualizations of the transport correspondences produced by the compared OT solvers. Appendix C studies the sensitivity of SWAP to randomness in the initialization and proposal directions. Finally, Appendix D presents pixel-level color-transfer examples.

\appendices
\section{Complete Proofs}
\label{app:supp-theory}

This appendix provides detailed proofs of the principal results presented in the main paper.
\subsection{Positive probability of proposing a swap pair}

\begin{proof}
Define
\[
f(\theta)
=
r_X(i;\theta)-r_Z(j;\theta).
\]
For every direction $\theta$ without projection ties, reversing the
projection direction reverses both rankings, and therefore
\[
f(-\theta)=-f(\theta).
\]

We first show that there exists a direction $\theta$ without projection
ties such that
\[
f(\theta)=0.
\]
Choose any direction $\theta_0$ without projection ties. If
\[
f(\theta_0)=0,
\]
the claim follows immediately. Otherwise, assume without loss of
generality that
\[
f(\theta_0)=m>0.
\]
Then
\[
f(-\theta_0)=-m<0.
\]

Since $d\geq2$, $\theta_0$ and $-\theta_0$ can be connected by a
generic continuous path on $\mathbb S^{d-1}$. Along this path, the
rank of $x_i$ changes only when
\[
\langle\theta,x_i-x_k\rangle=0,
\]
while the rank of $z_j$ changes only when
\[
\langle\theta,z_j-z_\ell\rangle=0.
\]
By the nonparallel condition, the path can be chosen so that only one
relevant rank-change event occurs at each crossing. Hence, when the
path crosses such an event, exactly one of the two ranks changes by
one, and therefore
\[
f\longmapsto f\pm1.
\]
Since $f$ changes from $m$ to $-m$ through integer unit steps, one of
the tie-free regions crossed by the path must satisfy
\[
f(\theta)=0.
\]
Thus, there exists a direction $\theta_*$ without projection ties such
that
\[
f(\theta_*)=0.
\]

It remains to show that such directions form an open set. Since
$\theta_*$ has no projection ties, all projected inequalities
determining the two rankings are strict. Hence they remain unchanged
in a sufficiently small neighborhood $\mathcal U$ of $\theta_*$. Therefore,
\[
f(\theta)=0
\qquad
\text{for all }\theta\in\mathcal U.
\]
Thus,
\[
\left\{
\theta\in\mathbb S^{d-1}:
r_X(i;\theta)=r_Z(j;\theta)
\right\}
\]
contains a nonempty open subset of $\mathbb S^{d-1}$ and consequently
has strictly positive probability under a uniformly random projection
direction.
\end{proof}

\subsection{Global optimality under the Monge condition}
    \begin{proof}
    It is well known that, for a Monge cost matrix, the identity
permutation is an optimal solution of the linear assignment problem
\cite{burkard1996perspectives}.

Let \(\sigma\in\Omega_2\). If \(\sigma\) is not the identity
permutation, then it contains an inversion, namely, there exist
indices \(i<k\) such that
\[
    \sigma(i)>\sigma(k).
\]
The Monge condition given in Eq.~\eqref{eq:monge-cost-matrix} implies that
\begin{equation}\label{eq:monge-inversion-swap}
    C_{i,\sigma(k)}+C_{k,\sigma(i)}
    \leq
    C_{i,\sigma(i)}+C_{k,\sigma(k)}.
\end{equation}
Since \(\sigma\in\Omega_2\), no pairwise exchange can strictly
decrease the matching cost. Hence,
\[
    C_{i,\sigma(k)}+C_{k,\sigma(i)}
    \geq
    C_{i,\sigma(i)}+C_{k,\sigma(k)}.
\]
Combining the two inequalities yields
\[
    C_{i,\sigma(k)}+C_{k,\sigma(i)}
    =
    C_{i,\sigma(i)}+C_{k,\sigma(k)}.
\]
Therefore, the inversion can be removed by a pairwise exchange
without changing the total cost.

Repeating this procedure finitely many times removes all inversions.
The resulting permutation \(\widetilde{\sigma}\) satisfies
\[
    \widetilde{\sigma}(1)
    <
    \widetilde{\sigma}(2)
    <
    \cdots
    <
    \widetilde{\sigma}(N).
\]
Since \(\widetilde{\sigma}\) is a permutation of
\(\{1,\ldots,N\}\), it must be the identity permutation. Moreover,
all exchanges used in this procedure preserve the total cost, and
therefore
\[
    C_c(\sigma)=C_c(\operatorname{id}).
\]
Since the identity permutation is globally optimal for a Monge cost
matrix, \(\sigma\) is also globally optimal. Hence,
\[
    \Omega_2\subseteq\Omega_N,
\]
which proves our proposition.
\end{proof}

\subsection{Global optimality under the anchored exchange-circulation condition}
    \begin{proof}
We first relate the symmetric part $s_{ij}$ to the pairwise swap
condition. Exchanging the assignments of $x_i$ and $x_j$ changes the
transport cost by
\begin{align*}
C_c(\sigma^{(i,j)})-C_c(\sigma)
&=
c(x_i,y_{\sigma(j)})
+
c(x_j,y_{\sigma(i)}) \\
&\quad
-
c(x_i,y_{\sigma(i)})
-
c(x_j,y_{\sigma(j)}) \\
&=
w_{ij}(\sigma)+w_{ji}(\sigma) \\
&=
2s_{ij}(\sigma).
\end{align*}
Since $\sigma\in\Omega_2$, no pairwise swap can decrease the
transport cost. Hence
\[
s_{ij}(\sigma)\ge0
\qquad
\text{for all } i,j.
\]

We now use the anchor $o$ to construct a scalar potential. Define
\[
p_i:=a_{oi}(\sigma).
\]
Since the antisymmetric part satisfies $a_{jo}=-a_{oj}$, the anchored
circulation can be written as
\begin{align*}
\kappa_{oij}(\sigma)
&=
a_{oi}(\sigma)+a_{ij}(\sigma)+a_{jo}(\sigma) \\
&=
p_i+a_{ij}(\sigma)-p_j.
\end{align*}
Therefore,
\[
a_{ij}(\sigma)
=
p_j-p_i+\kappa_{oij}(\sigma).
\]
Combining this identity with $w_{ij}=s_{ij}+a_{ij}$ gives
\begin{align*}
w_{ij}(\sigma)+p_i-p_j
&=
s_{ij}(\sigma)
+
a_{ij}(\sigma)
+
p_i-p_j \\
&=
s_{ij}(\sigma)+\kappa_{oij}(\sigma).
\end{align*}
By the anchored exchange-circulation condition
given in Eq.~(6) of the main paper,
\begin{equation}
\label{eq:reduced-nonnegative}
w_{ij}(\sigma)+p_i-p_j\ge0
\qquad
\text{for all } i,j.
\end{equation}

We next show that every cyclic reassignment has nonnegative cost
change. Consider an arbitrary cycle of distinct indices
\[
\gamma=(i_1,i_2,\ldots,i_k),
\qquad
i_{k+1}=i_1,
\]
and let $\sigma^\gamma$ denote the permutation obtained by moving the
target currently assigned to $x_{i_r}$ to $x_{i_{r+1}}$, namely,
\[
\sigma^\gamma(i_{r+1})=\sigma(i_r),
\qquad
r=1,\ldots,k.
\]
The corresponding change in transport cost is
\begin{align*}
C_c(\sigma^\gamma)-C_c(\sigma)
&=
\sum_{r=1}^k
\left[
c(x_{i_{r+1}},y_{\sigma(i_r)})
-
c(x_{i_r},y_{\sigma(i_r)})
\right] \\
&=
\sum_{r=1}^k
w_{i_r i_{r+1}}(\sigma).
\end{align*}
Using \eqref{eq:reduced-nonnegative}, we obtain
\begin{align*}
C_c(\sigma^\gamma)-C_c(\sigma)
&=
\sum_{r=1}^k
\left[
w_{i_r i_{r+1}}(\sigma)
+
p_{i_r}
-
p_{i_{r+1}}
\right] \\
&\qquad
-
\sum_{r=1}^k
\left(
p_{i_r}-p_{i_{r+1}}
\right) \\
&\ge0,
\end{align*}
because the potential terms telescope,
\[
\sum_{r=1}^k
\left(
p_{i_r}-p_{i_{r+1}}
\right)
=0.
\]
Therefore, no cyclic reassignment of any length can decrease the
transport cost.

Finally, any permutation relative to $\sigma$ can be decomposed into
a collection of disjoint cycles. Since each such cycle has
nonnegative cost change, no permutation has a lower transport cost than
$\sigma$. Hence $\sigma$ is globally optimal, and therefore
\[
\sigma\in\Omega_N.
\]
Moreover, the construction above provides an explicit optimality
certificate: the anchor-based potential
$p_i:=a_{oi}(\sigma)$ satisfies
\[
w_{ij}(\sigma)+p_i-p_j\ge 0
\qquad
\text{for all } i,j.
\]
\end{proof}

\section{Geometric Transport Visualization}
\label{app:supp-permutation}
We visualize the geometric correspondence represented by the output of each OT solver. We use a two-dimensional checkerboard with $N=2^{10}$ and seed $200$. For a permutation $\sigma$ returned by Exact, SWAP and HiRef, every segment $x_i\mapsto y_{\sigma(i)}$ is drawn.
For a coupling $P$ returned by Sinkhorn and HALO, we instead show the row-wise barycentric map
$$
    x_i \mapsto \bar y_i(P),
    \qquad
    \bar y_i(P)=\frac{\sum_j P_{ij}y_j}{\sum_j P_{ij}}.
    \label{eq:supp-barycentric-map}
$$
The barycentric endpoints summarize where each source row sends mass, but need not coincide with target samples and do not themselves define a feasible
permutation. All panels use the same source and target samples.

Figure~\ref{fig:supp-transport-visualization} shows that the SWAP permutation follows essentially the same block-to-block transport pattern as the exact
assignment while preserving a one-to-one map. Sinkhorn produces visibly interior barycentric endpoints, reflecting mass spread across several targets, whereas HALO produces a highly concentrated coupling on this instance and therefore appears close to a permutation under the same visualization.  HiRef remains permutation-valued but exhibits a different set of long correspondences. 

\begin{figure*}[t]
\centering
\includegraphics[width=0.98\textwidth]{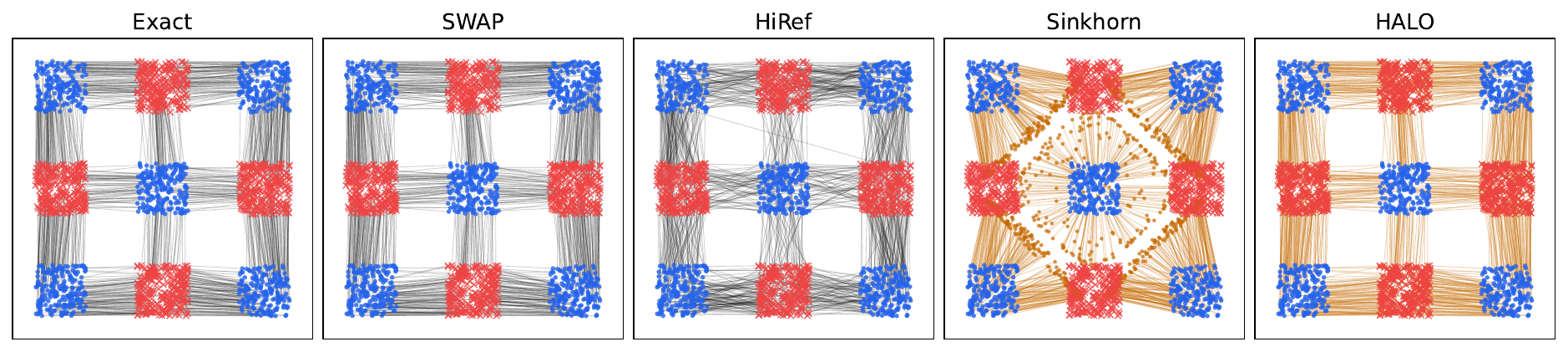}
\caption{Geometric transport visualization. Blue points denote the source samples and red crosses denote the target samples. The orange points indicate the barycentric projections.}
\label{fig:supp-transport-visualization}
\end{figure*}

\section{Sensitivity to Algorithmic Randomness}
\label{app:supp-randomness}

SWAP uses randomness in both the initialization and the subsequent
proposal directions.  We isolate these two sources on the Brenier-map ground-truth problem in $d=64$ at $N=2^{16}$.  Each study contains $50$
runs with $100{,}000$ sliced directions. 

\subsection{Randomness in the initialization}
\label{app:supp-initialization}
The first study examines randomness in the initialization by varying the
random projection direction used to construct the initial permutation, while
keeping the subsequent sequence of proposal directions fixed. As shown in Fig.~\ref{fig:supp-init-randomness}, SWAP rapidly suppresses most
of the variation introduced by the random initialization while continuing to
decrease the objective. The early inset highlights the initial sensitivity on
a relative scale, whereas the inset over the final $1,000$ sliced directions shows that
the trajectories have collapsed into a very narrow band. This demonstrates
that initialization-induced differences are progressively attenuated during
optimization and have only a limited effect on the late-stage objective.

\begin{figure}[H]
\centering
\includegraphics[width=0.86\columnwidth]{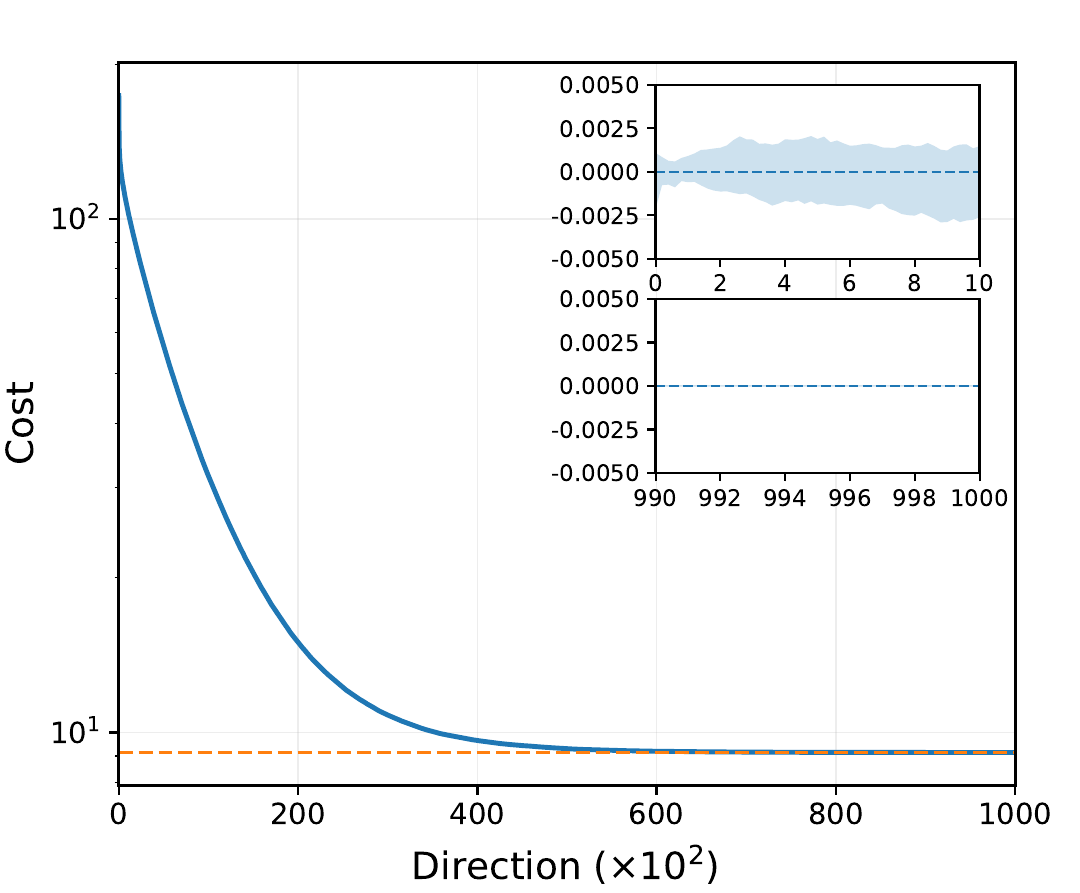}
\caption{Sensitivity to the initialization. The solid curve is the pointwise median cost and the shaded region shows the minimum--maximum range.
The upper and lower insets report this range relative to the pointwise median.}
\label{fig:supp-init-randomness}
\end{figure}

\subsection{Randomness in the sliced proposal}
\label{app:supp-schedule}

The second study fixes the initial permutation and varies only the sequence of sliced proposal directions.  Figure~\ref{fig:supp-schedule-randomness} shows
the same qualitative behavior: proposal-dependent differences are visible during the transient, but they remain controlled and contract as the objective approaches the late-stage plateau. The plotting conventions are the same as in Figure ~\ref{fig:supp-init-randomness}. Different proposal paths can explore the same problem differently while reaching essentially the same objective
neighborhood.

\begin{figure}[H]
\centering
\includegraphics[width=0.86\columnwidth]{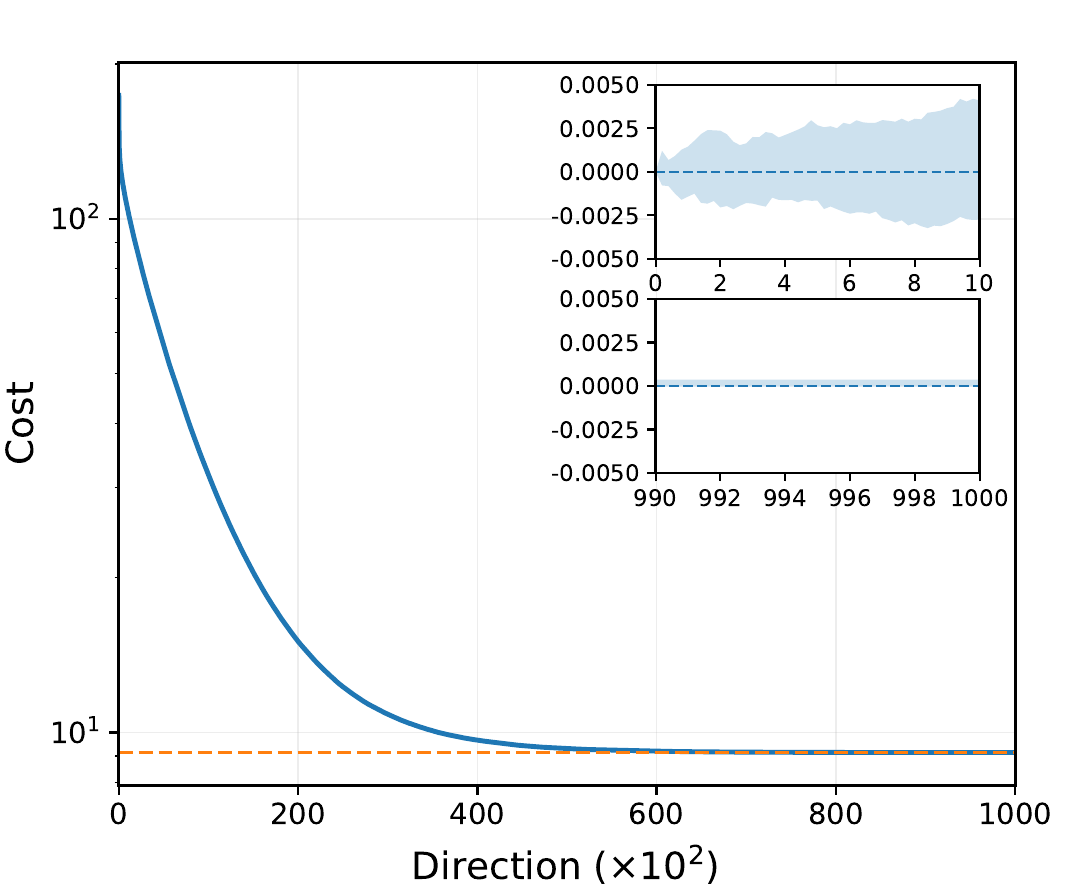}
\caption{Sensitivity to the sliced proposal directions.}
\label{fig:supp-schedule-randomness}
\end{figure}
\FloatBarrier

\section{Pixel-Level Color Transfer}
\label{app:supp-color-transfer}

SWAP can be applied directly to pixel-level color transfer at image resolution. 
Each image is resized to \(512\times512\), yielding
\(N=2^{18}\) pixels. Following the chroma-transfer formulation~\cite{solomon2015convolutional}, we
represent each source pixel by its CIE-Lab chromatic coordinates
\(u_i=(a_i^*,b_i^*)\), and each reference pixel by \(v_j\). We then solve the
permutation problem
$$
\min_{\sigma\in S_N}
\frac{1}{N}\sum_{i=1}^N
\lVert u_i-v_{\sigma(i)}\rVert_2^2.
$$
The transferred image preserves the source luminance \(L_i^*\) and spatial
arrangement, while assigning to each source pixel the reference chroma
\(v_{\sigma(i)}\). Hence the source determines image structure and luminance,
whereas the reference contributes a one-to-one chroma palette. The computation is
performed directly on the pixel sets and does not require forming an
\(N\times N\) cost matrix. Representative results are shown in
Fig.~\ref{fig:supp-color-transfer}.

\begin{figure*}[t]
\centering
\includegraphics[width=0.96\textwidth,height=0.25\textheight,keepaspectratio]{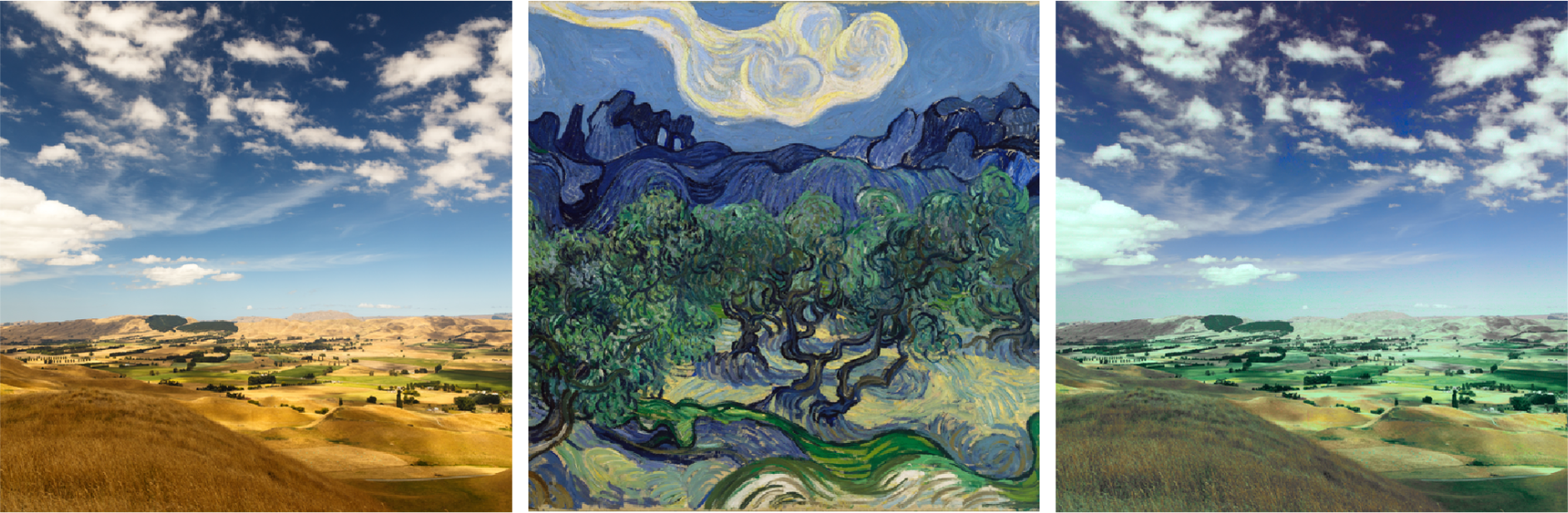}\\[0.5em]
\includegraphics[width=0.96\textwidth,height=0.25\textheight,keepaspectratio]{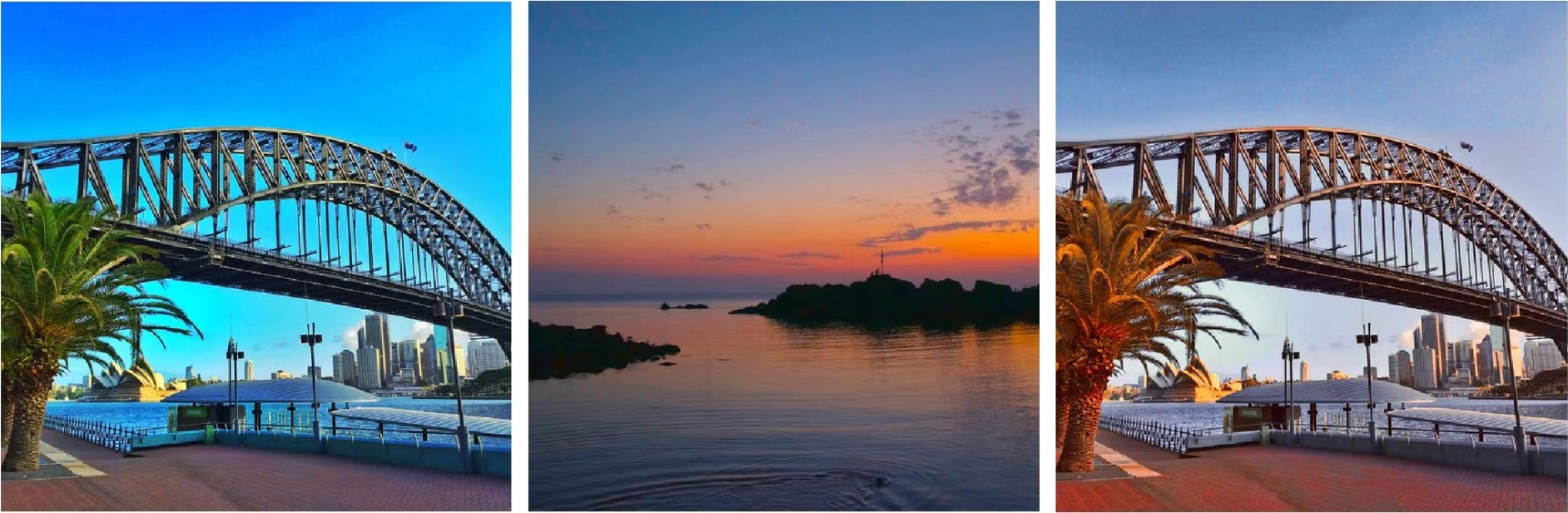}\\[0.5em]
\includegraphics[width=0.96\textwidth,height=0.25\textheight,keepaspectratio]{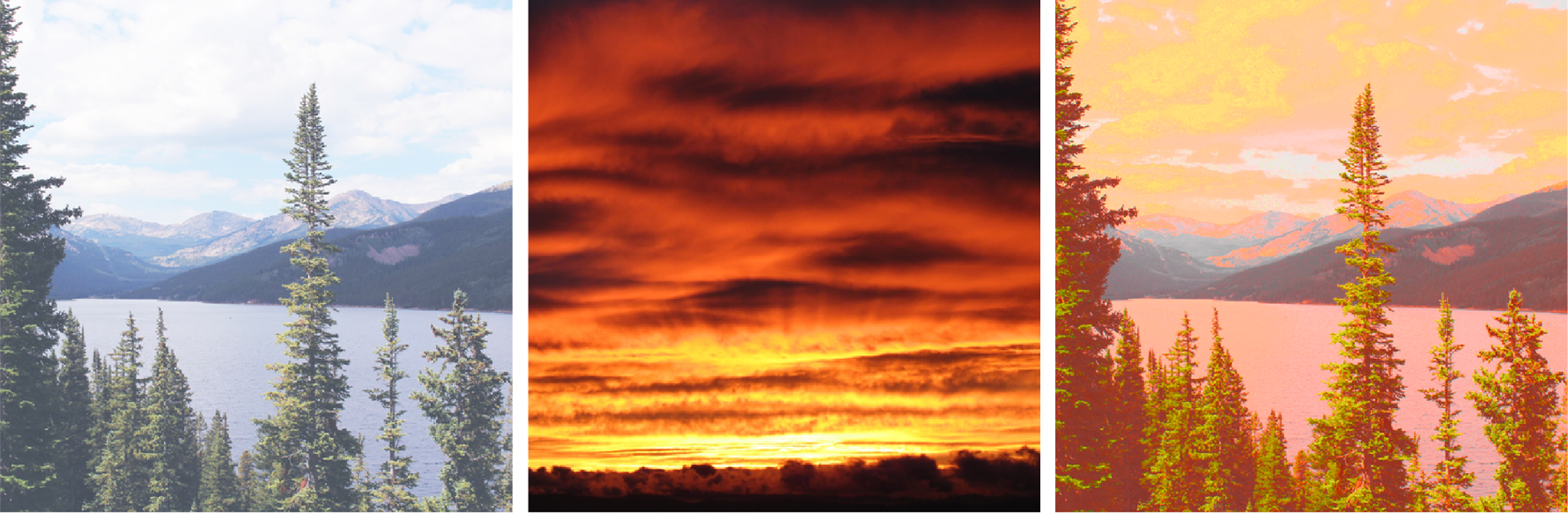}\\[0.5em]
\caption{Selected \(512\times512\) pixel-level color-transfer examples. In each row, the panels show, from left to right, the source image, the reference image, and the transferred result.}
\label{fig:supp-color-transfer}
\end{figure*}

\clearpage
\twocolumn
\bibliographystyle{IEEEtran}
\bibliography{refs}
\end{document}